\documentclass[11pt]{article}
\usepackage{amsfonts,eucal}
\usepackage{amssymb,latexsym}
\usepackage{amsbsy,amsmath}
\newtheorem{theorem}{Theorem}[section]
\newtheorem{lemma}[theorem]{Lemma}
\newtheorem{corollary}[theorem]{Corollary}
\newtheorem{proposition}[theorem]{Proposition}
\newtheorem{itdefinition}[theorem]{Definition}
\newenvironment{definition}{\begin{itdefinition}\rm}{\end{itdefinition}}

\newenvironment{proof}[1][Proof]{\par\noindent{\em #1}.
}{\hfill\framebox(6,6) \par\medskip}

\newcommand{\Irr}{\mathop{\rm Irr}\nolimits}
\newcommand{\Fr}{\mathop{\rm Fr}\nolimits}

\newcommand{\sgn}{\mathop{\rm sgn}\nolimits}
\renewcommand{\mod}{\mathop{\rm mod}\nolimits}
\newcommand{\pdeg}{\mathop{\rm pdeg}\nolimits}
\newcommand{\wdeg}{\mathop{\rm wdeg}\nolimits}

\begin{document}

\title{Modular representations
of the special linear  groups with small weight multiplicities}

\author{A.A.~Baranov\\
\small\it Department of  Mathematics, University of Leicester,  \\
\small\it Leicester, LE1 7RH, UK \\
\small\it ab155@le.ac.uk\\
\\
A.A.~Osinovskaya\\
\small\it Institute of Mathematics, National Academy of Sciences of Belarus,\\
\small\it 11 Surganov street, Minsk, 220072, Belarus\\
\small\it anna@im.bas-net.by\\
\\
I.D.~Suprunenko\thanks{
The second and the third authors were
supported by the Institute of Mathematics of the National Academy
of Sciences of Belarus in the framework of the State Basic
Research Programme ``Mathematical Models'' (2006--2010) and the
State Research Programme ``Convergence'' (2011--2015) and partially supported
by the Belarus Basic Research Foundation, Project F04-242.}\\
\small\it Institute of Mathematics, National Academy of Sciences of Belarus,\\
\small\it 11 Surganov street, Minsk, 220072, Belarus\\
\small\it suprunenko@im.bas-net.by}

\date{}

\maketitle

\begin{center} 
\small\it 
Dedicated with admiration to A.E.\ Zalesski on the occasion of his $75$th birthday
\end{center}

\medskip

\begin{abstract}
We classify irreducible representations of the
special linear  groups in positive characteristic
with small weight multiplicities with
respect to the group rank and give
estimates for the maximal weight multiplicities.
For the natural embeddings of the classical groups, inductive systems of representations with totally bounded weight multiplicities are classified. An analogue of the Steinberg tensor product theorem for arbitrary indecomposable inductive systems for such embeddings is proved.
\end{abstract}

\section{Introduction}

In what follows $K$ is an algebraically closed field of
characteristic $p>0$;
$G_n$ is a classical
algebraic group of rank $n$ over $K$;
$\Irr G_n$ is the set of all
rational irreducible representations (or simple modules) of $G_n$
up to equivalence, $\Irr^p G_n\subset \Irr G_n$ is the subset of
$p$-restricted ones;
$\Irr M\subset \Irr G_n$ is the set of composition factors of a module
$M$ (disregarding the multiplicities),
$\omega(M)$ is the highest weight of a simple module $M$;
$L(\omega)$ is the simple
$G_n$-module with highest weight $\omega$; $\omega_1^n, \ldots,
\omega_n^n$ are the fundamental weights of $G_n$;
$\omega_0^n=\omega_{n+1}^n=0$ by convention. A weight
$\sum_{i=1}^n a_i\omega_i^n$ is $p$-restricted if all $a_i<p$. By the {\em weight degree} of
a module $M$ we mean the maximal dimension of the weight subspaces
in $M$, i.e.
$$
\wdeg M=\max_{\mu\in\Lambda(M)}\dim M^\mu
$$
where $\Lambda(M)$ is the set of weights of $M$.
In particular, we say that $M$ has a small weight degree if $\wdeg M$ is small with respect to $n$.

For the classical algebraic groups modular representations  of
weight degree 1 were classified in~\cite{Seitz,ZSVesti}. To state
the result, first define the following sets of weights of the
group $G_n=A_n(K)$, $B_n(K)$, $C_n(K)$, or $D_n(K)$:
\begin{eqnarray*}
\Omega_p(A_n(K))&=&
\{0,\omega_k^n,
(p-1-a)\omega_k^n+a\omega_{k+1}^n\mid 0\leq k\leq n,\  0\leq a\leq p-1 \},
\\
\Omega_p(B_n(K))&=&
\{0, \omega^n_1, \omega^n_n\}, \\
\Omega_p(C_n(K))&=&
\{0, \omega_1^n, \frac{p-1}{2}\omega_n^n,
  \omega_{n-1}^n+\frac{p-3}{2}\omega_n^n \} \quad (p>2),\\
\Omega_p(D_n(K))&=&
\{0,\omega_1^n, \omega_{n-1}^n, \omega_n^n\}, \\
\Omega(G_n)&=&\{\sum_{j=0}^k p^j\lambda_j\mid
k\ge0,\ \lambda_j\in\Omega_p(G_n)\}.
\end{eqnarray*}

\begin{theorem}[{\cite[6.1]{Seitz}, \cite[Proposition 2]{ZSVesti}}] \label{one}
Let $G_n$ be a classical algebraic group of rank $n\ge 4$
and let $M$ be a rational simple $G_n$-module. Assume $p>2$ for $G=B_n(K)$ or $C_n(K)$. Then $\wdeg M =1$ if and only if $\omega(M)\in\Omega(G_n)$.
\end{theorem}

Obviously, a simple module $M$ is $p$-restricted with $\wdeg M =1$ if and only if
$\omega(M)\in\Omega_p(G_n)$. The $A_n(K)$-modules
$L((p-1-a)\omega^n_k+a\omega^n_{k+1})$ are truncated symmetric powers
of the natural module~\cite[Proposition 1.2]{ZS}. Thus, the only $p$-restricted
modules of weight degree 1 for type $A$ are the fundamental
modules and truncated symmetric powers of the natural module. Recall that $B_n(K)\cong C_n(K)$ for $p=2$ (as abstract groups). So we do not consider groups of type $B_n$ in characteristic 2. For groups of type $C_n$ in this case the description of irreducible modules of weight degree 1 is more involved (see details in Section \ref{inductive2}).

In this paper we classify irreducible representations of the special linear groups of small weight degree. For other classical groups this was done by the authors earlier. In particular, it was shown that for these groups and odd $p$ no irreducible modules $M$ exist with $1<\wdeg M < n-7$.

\begin{theorem}[{\cite[Theorem 1.1]{BCD}, \cite[Theorem 1]{Dchar2},\cite[Theorem 1]{C_small_char}}]
\label{BCD} Let $n\geq 8$ and let $G_n=B_n(K)$, $C_n(K)$ or
$D_n(K)$. Let $M$ be a rational simple $G_n$-module with
$\omega(M)\notin\Omega(G_n)$. Suppose that $p>2$ for $G_n=B_n(K)$
or $C_n(K)$. Then $\wdeg M \geq
n-4-[n]_4$ where $[n]_4$ is the residue of $n$ modulo $4$. In
particular, $\wdeg M\geq n-7$.
\end{theorem}

The main case ($p>2$ for $G_n=B_n(K)$ or $D_n(K)$ and $p>7$ for $G_n=C_n(K)$) was settled in \cite{BCD}; \cite{Dchar2} deals with type $D$ for $p=2$; and \cite{C_small_char} gives a new proof for type $C$ for all $p$. For $G=C_n(K)$ and $p=2$ a new exceptional series of modules with $\wdeg=2^s$ appears (see details in Section \ref{inductive2}).

Now  assume that $G_n=A_n(K)$.  Let $M\in\Irr G_n$, $\omega(M)=a_1\omega^n_1+\ldots+a_n\omega^n_n$, and $M^*$ be the dual of $M$. Note that
$\omega(M^*)=a_n\omega^n_1+a_{n-1}\omega^n_2 +\ldots+a_1\omega^n_n$ and $\wdeg M=\wdeg M^*$. Define
the {\em polynomial degree} of $M$ as the polynomial degree of the corresponding
polynomial representation of $GL_{n+1}(K)$, i.e.
\begin{equation}
\label{pdegg}
\pdeg M=\sum_{k=1}^n ka_k.
\end{equation}
Denote by $V_n$ the natural module
for $G_n$. Note that every simple module of polynomial degree $d$ can be
obtained as a composition factor of the $d$th tensor power
$V_n^{\otimes d}$. More exactly, we have the following. Set
\begin{equation}
\label{Ldn}
\mathcal{L}^d_n=\cup_{j\leq d} \Irr  V_n^{\otimes j}, \quad
\mathcal{R}^d_n=\cup_{j\leq d} \Irr ( V_n^*)^{\otimes j}.
\end{equation}
Then $\mathcal{L}_n^d=\{ M\in \Irr G_n \mid \pdeg M\leq d \}$ and
$\mathcal{R}_n^d=\{ M\in \Irr G_n \mid \pdeg M^* \leq d \}$
(Proposition \ref{LR}). For $d\le n$, it is not difficult to see
that $\wdeg V_n^{\otimes d}=d!$ (Lemma \ref{wdegVd}). This means
that modules of small polynomial degree $d$ (with, say, $d!<n$)
have small weight degree ($<n$), which gives many more small
weight degree modules for  type $A$ in addition to those described
in Theorem \ref{one}. This makes situation more difficult than in
the case of other classical groups, especially for non
$p$-restricted modules. Our first main result describes
$p$-restricted irreducible representations of the special linear
groups of small weight degree.

\begin{theorem}
\label{main1}
Let $M\in \Irr^p A_n(K)$ and $d=\min\{\pdeg M, \pdeg M^*\}$. Assume
$\omega(M)\not\in \Omega_p(A_n(K))$. Then the following hold.
\begin{itemize}
\item[$(i)$] If $n\geq16$ and $d>n$, then
$$\wdeg M>\sqrt{n}/p -1.$$
\item[$(ii)$] If $d\le n$, then
$$d-2\le \wdeg M\le d!.$$
Moreover, $M\cong L(a_1\omega^n_1+\ldots+a_d\omega^n_d)$ or
$L(a_d\omega^n_{n-d+1}+\ldots+a_1\omega^n_n)$ with $a_1+2a_2+\dots+da_d=d$,
and $\wdeg M$
is determined by the sequence $(a_1,\ldots,a_d)$ only and
does not depend on $n$.
\end{itemize}
In particular, if $n\geq 16$ and $\wdeg M\le \sqrt{n}/p -1$, then $M$
is as in part $(ii)$ with $d\le \sqrt{n}/p +1$.
\end{theorem}

The $\sqrt{n}/p -1$ estimate in part $(i)$ was obtained by
applying the Schur functor. It is a quick and rough estimate and can
probably be improved if one uses a more thorough analysis, similar
to that of \cite{BCD}. One should expect something close to $n$,
as in Theorem \ref{BCD}. Unfortunately, this seems to be very
difficult to obtain at the moment as too many modules of small
weight degree exist for  type $A$ and the methods used in
\cite{BCD} fail to work. But our estimate is good enough to
identify the modules with small weight degree and get a full
classification of the inductive systems of representations for
$A_\infty$ with bounded weight multiplicities (see below).

In what follows for all classical groups
$\Fr$ is the Frobenius morphism of $G_n$ associated with raising
the elements of $K$ to the $p$th power; $M^{[k]}$ denotes a
$G_n$-module $M$ twisted by the $k$th power of $\Fr$.
Let $M\in\Irr G_n$. Assume that $\omega(M)=\sum^s_{k=0} p^k\lambda_k$ with $p$-restricted dominant weights $\lambda_k$ of $G_n$. Put $M_k=L(\lambda_k)$. By the Steinberg tensor product
theorem~\cite{Stein},
\begin{equation}
\label{Fr} M\cong\otimes^s_{k=0}M_k^{[k]}.
\end{equation}
It is obvious that $\wdeg M\geq \wdeg M_0\cdot\ldots\cdot \wdeg M_s$ (Lemma~\ref{degtenz}).
Therefore, the question of describing non $p$-restricted
$G_n$-modules of small weight degree is essentially reduced to combining various Frobenius twists
of $p$-restricted modules of small weight degree and making sure that the weight degree does not become too large
(see Corollary \ref{1.5}, Theorem \ref{CnonpA}, and Proposition \ref{tenz}).

Note that the results above can be
considered as a modular analogue of the following problem solved by Mathieu
\cite{mathieu}: describe all infinite dimensional weight modules with bounded
weight multiplicities for a finite dimensional simple Lie algebra over $\mathbb{C}$.
Some particular cases, including so-called completely pointed modules (i.e. with one dimensional
weight spaces) were previously considered in \cite{BBL,BL,F}.
It is interesting to note that by specializing $p$ to $0$ in the weights in the set $\Omega_p(G_n)$
we get highest weights of completely pointed modules
(e.g. $(-1-a)\omega^n_k+a\omega^n_{k+1}$ for type $A_n$ and
$\omega^n_{n-1}-\frac{3}{2}\omega^n_n$ and $-\frac{1}{2}\omega^n_n$ for type $C_n$).

Estimates of weight multiplicities obtained above can be used for
recognizing linear groups containing matrices with small
eigenvalue multiplicities. Indeed, it occurs that  only for some
special classes of representations of simple classical algebraic
groups, their images can contain matrices all whose eigenvalue
multiplicities are small enough with respect to the group rank.

At the end of the paper we classify inductive systems of representations
with bounded weight multiplicities for the natural embeddings of the classical groups. In what follows $\mathbb{N}$ is the set of positive integers.
For a group $G$, a subgroup $H\subset G$ and a $G$-module $M$ denote by
$M{\downarrow}H$ the restriction of $M$ to $H$.
Let
\begin{equation}
\label{seq} \Gamma_1 \subset \Gamma_2 \subset \dots \subset \Gamma_n
\subset \dots
\end{equation}
be a chain of fixed embeddings of algebraic groups $\Gamma_n$ over $K$ and let
$\Phi_n$, $n\in \mathbb{N}$, be a nonempty finite subset of $\Irr \Gamma_n$, for each $n$.
Recall that the system $\Phi=\{ \Phi_n \mid n \in \mathbb{N} \}$ is called an
{\em inductive system} of representations (or modules) for (\ref{seq}) if
\[
\bigcup_{\varphi\in\Phi_{n+1}} \Irr(\varphi{\downarrow}\Gamma_n)=\Phi_n
\]
for all $n\in\mathbb{N}$.
Inductive systems have been introduced by A. Zalesskii in~\cite{Zal}.
They can be regarded as an asymptotic version of the branching
rules for the embeddings (\ref{seq}). Observe that in positive characteristic one cannot expect to find explicit analogues of the classical branching rules in characteristic 0 which have quite a lot of applications, so their asymptotic versions can be useful.
Moreover, inductive systems can be applied to the study
of ideals in group algebras of locally finite groups. It is proved
in \cite{ZSt} that there exists a bijective correspondence between
the inductive systems for a locally finite group and the
semiprimitive ideals of the corresponding group algebra. So far we
know little about the structure of inductive systems. Minimal and
minimal nontrivial inductive systems of modular representations for
natural embeddings of algebraic and finite groups of type $A_n$ were
classified in~\cite{BarSupr}. For other classical groups the
question on the minimal inductive systems seems substantially more
difficult. For natural embeddings of symplectic groups in positive
characteristic examples of such systems that have no analogues in
the characteristic 0 case were constructed in~\cite{ZSVesti}
and~\cite{BS}.

Let $\alpha_1,\ldots,\alpha_n$ be the simple roots of $G_n$ labeled as
in~\cite{Burb} (it will always be clear from the context what group is considered). It is well known that the root subgroups associated with the roots $\pm\alpha_{n-k+1},\ldots,\pm\alpha_n$ generate a subgroup isomorphic to $G_k$. If we identify $G_k$ with this subgroup, we obtain a sequence of natural embeddings
\begin{equation}
\label{seq1}
G_1\subset G_2\subset \ldots \subset G_n \subset \ldots.
\end{equation}
In this paper we consider only inductive systems for the sequence~(\ref{seq1}).

\begin{definition}
Let $\Phi$ be an inductive system of representations. We say that $\Phi$ is a
{\em BWM-system} (bounded weight multiplicities system) if there exists
$m\in\mathbb{N}$ such that $\wdeg\varphi\le m$ for all $\varphi\in\Phi_n$ and
all $n$.
For a BWM-system $\Phi$ we define $\wdeg\Phi=\max_{\varphi\in\Phi}
\wdeg\varphi$.
\end{definition}

In Sections~\ref{inductive} and~\ref{inductive2} we classify all BWM-systems for all four
types of classical groups. To state the main results,  we need to
introduce some notation.
For any dominant weight $\omega$ of $G_n$ denote by $\delta(\omega)$
the value of $\omega$ on the maximal root of the root system of
$G_n$. For a simple module $M\cong L(\omega)$ put
$\delta(M)=\delta(\omega)$.
Let $T\subset\mathbb{N}$ be infinite.
Assume that $R_t\subset\Irr G_t$ is nonempty for each $t\in T$ and
that there exists $k\in\mathbb{N}$ such that $\delta(M)<k$ for all
$M\in R_t$ and for all $t$. Denote by $\Pi_n$ the set of all
$G_n$-modules $Q$ such that $Q$ is a composition factor of the
restriction $Y{\downarrow}G_n$ for some $t>n$, $t\in T$, and $Y\in
R_t$. Assume that $R_t\subset \Pi_t$ for all $t$. By Lemma~\ref{par3},  $\Pi=\{\Pi_n \mid n\in\mathbb{N}\}$ is an inductive system for the groups $G_n$. We will write
$\Pi=\langle R_t\mid t\in T\rangle$ and call $\Pi$ the {\em
inductive system generated by $R_t$}. If every $R_t$ consists of a
single module $Y_t$, we use a simplified notation $\Pi=\langle
Y_t\mid t\in T\rangle$. Let $\Phi$ be an inductive system. We say
that $\Phi$ is a {\em $p$-restrictedly generated system} if
$\Phi=\langle \Lambda_t\mid t\in T\rangle$ with $\Lambda_t\subset
\Irr^p G_t$ for all $t\in T$.

For arbitrary inductive systems $\Phi$ and $\Psi$ define the
collections $\Fr(\Phi)$ and $\Phi\otimes\Psi$ in a natural way:
\begin{eqnarray*}
\Fr(\Phi)_n &=& \{ \varphi^{[1]} \mid \varphi\in\Phi_n\},\\
(\Phi\otimes\Psi)_n &=& \bigcup_{\varphi\in\Phi_n,\ \psi\in\Psi_n}
\Irr (\varphi\otimes\psi).\\
\end{eqnarray*}
By Lemma~\ref{pten}, $\Fr(\Phi)$ and $\Phi\otimes\Psi$ are inductive
systems. The union  of inductive systems $\Phi$
and $\Psi$ and the inclusion relation for such systems are defined
in a natural way. An inductive system $\mathcal{T}$ is called
{\em decomposable} if $\mathcal{T}$ is the union of inductive systems
$\Phi$ and $\Psi$ that do not coincide with $\mathcal{T}$, and
{\em indecomposable} otherwise.
For an inductive system
$\Phi$ put
\[
\delta(\Phi_n)=\max\{\delta(\varphi) \mid \varphi\in\Phi_n \}.
\]
Then $\delta(\Phi_n)$ does not depend on $n$ (Lemma~\ref{delta}), so we can define $\delta(\Phi)$ as $\delta(\Phi_n)$.

In Section \ref{steinberg}  we prove the following analogue of the Steinberg product theorem
 for inductive systems, which is of independent interest.

\begin{theorem}
\label{anSt} Let $\Phi$ be an indecomposable inductive system for the sequence~(\ref{seq1}). 
Then there exist $p$-restrictedly generated inductive systems $\Phi^j$, $0\leq j\leq k$, such that
$\Phi=\otimes_{j=0}^k \Fr^j(\Phi^j)$.
\end{theorem}

Now assume that $G_n=A_n(K)$. Recall the sets $\mathcal{L}_n^d$ and $\mathcal{R}_n^d$ defined in  (\ref{Ldn}).
Lemma~\ref{ind} implies that
$\mathcal{L}^d=\{ \mathcal{L}_n^d\mid n\in\mathbb{N} \}$ and
$\mathcal{R}^d=\{\mathcal{R}_n^d\mid n\in\mathbb{N} \}$
are
inductive systems. Note that $\mathcal{L}_n^1=\{L(0),  V_n\}$.
Set
\begin{eqnarray}
\mathcal{F}_n&=&\{ L(\omega_0^n), L(\omega_1^n), \ldots,
L(\omega_n^n)\}, \label{Fn} \\
\mathcal{T}_n&=&\{ L((p-a-1)\omega_i^n+a\omega_{i+1}^n) \mid 0\leq
a<p, \quad 0\leq i\leq n\} \label{Tn}
\end{eqnarray}
($\omega_{n+1}^n$ is treated as $0$). By Lemma~\ref{ind},  $\mathcal{F}=\{ \mathcal{F}_n \mid
n\in\mathbb{N}\}$ and $\mathcal{T}=\{ \mathcal{T}_n \mid
n\in\mathbb{N}\}$ are inductive systems. Note that the
representations of $\mathcal{T}$ are realized exactly in the truncated symmetric powers of the natural module.

Let $d\in \mathbb{N}$. Fix any integers $a_i\ge 0$ for $0\leq i\leq d$. For $n\geq d$ let
$M_{n,L}(a_1,\ldots,a_d)$ be a simple $G_n$-module with highest
weight $a_1\omega^n_1+\ldots+a_d\omega^n_d$ and
$M_{n,R}(a_1,\ldots,a_d)$ be a simple $G_n$-module  with highest
weight $a_d\omega^n_{n-d+1}+\ldots+a_1\omega^n_n$. Set
\begin{eqnarray*}
C_L(a_1,\ldots,a_d)&=&\langle M_{n,L} (a_1,\ldots, a_d)\mid n\geq
d\rangle,\\
C_R(a_1,\ldots,a_d)&=&\langle M_{n,R} (a_1,\ldots,
a_d)\mid n\geq d\rangle.
\end{eqnarray*}
By Lemma~\ref{lLR}, the systems $C_L(a_1,\ldots,a_d)$ and $C_R(a_1,\ldots,a_d)$ are well defined.

\begin{theorem}
\label{t1A} Let $G_n=A_n(K)$. Assume that $\Phi$ is a
$p$-restrictedly generated indecomposable BWM-system. Then
$\Phi=\mathcal{F}$, $\mathcal{T}$, $C_L(a_1,\ldots,a_d)$ or  $C_R(a_1,\ldots,a_d)$
for some integers $a_1,\ldots,a_d<p$.
\end{theorem}

Let $\Phi$ be an inductive system. Assume that
\[
\Phi=\otimes_{k=0}^s \Fr^k (\Phi^k),
\]
where $\Phi^k$ are $p$-restrictedly generated systems. We say that
$\Phi$ is {\em special} if each $\Phi^k$ is equal to one of the systems $C_L(a_1,\ldots,a_d)$,
$C_R(a_1,\ldots,a_d)$,  $\mathcal{F}$, or $\mathcal{T}$.

Let $\Phi$ be special. Then for every $k$,  either $\Phi^k=
\mathcal{F},\mathcal{T}$ or there exists $d$ such that
$\Phi^k\subset \mathcal{L}^d$ or $\mathcal{R}^d$. Therefore,
$\Phi$ can be represented in the form
\[
\Phi=\Psi^0\otimes\ldots\otimes\Psi^l
\]
with
\begin{equation}
\label{Fr1ind} \Psi^f=\otimes_{k=i_{f-1}+1}^{i_f} \Fr^k (\Phi^k),
\end{equation}
where the indices $i_f$, $0\leq f \leq l$,  satisfy the following: $i_{-1}=-1$ and for each
$f$,  either all $\Phi^k$ have the form $C_L(a_1,\ldots,a_d)$ for
$i_{f-1}+1\leq k\leq i_f$, or all $\Phi^k$ have the form $C_R(a_1,\ldots,a_d)$ for $i_{f-1}+1\leq k\leq i_f$, or $i_{f-1}+1=k=i_f$ and $\Phi^k=\mathcal{F}$ or $\mathcal{T}$.
Fix minimal $l$ with this property. Then the systems $\Psi^f$ are uniquely determined.

\begin{theorem}
\label{a} Let $G_n=A_n(K)$. Indecomposable BWM-systems are exhausted
by special inductive systems with the following property
$\delta(\Psi^f)<p^{i_f+1}$ for all $\Psi^f$ with $f<l$ ($i_f$ are
such as in~$(\ref{Fr1ind})$). An arbitrary
BWM-system is a finite union of indecomposable ones.
\end{theorem}

Theorems~\ref{BCD} and~\ref{12new} allow us to find the BWM-systems for the remaining
series of classical groups.
Put
\[
\mathcal{S}_n=\left\{
\begin{array}{ll}
\{L(\omega_n^n)\} &\mbox{ for } G_n=B_n(K),\\ \{L(\omega_{n-1}^n),
L(\omega_n^n)\} &\mbox{ for } G_n=D_n(K),\\
\{L(\frac{p-1}{2}\omega_n^n),
L(\omega_{n-1}^n+\frac{p-3}{2}\omega_n^n )\} & \mbox{ for }
G_n=C_n(K), p>2
\end{array}
\right.
\]
and $\mathcal{L}_n=\{L(0),L(\omega_1^n)\}$.
Lemmas~\ref{l*} and~\ref{ind2} imply that $\mathcal{L}=\{ \mathcal{L}_n\mid n \in \mathbb{N}  \}$ and $\mathcal{S}=\{ \mathcal{S}_n\mid n \in \mathbb{N}  \}$
are inductive systems. Obviously,  the collection
$\mathcal{O}=\{\mathcal{O}_n\mid n \in \mathbb{N}  \}$ with $\mathcal{O}_n=\{L(0)\}$ is an
inductive system for all types.

\begin{theorem}
\label{t1} Let $G_n=B_n(K)$, $C_n(K)$ or  $D_n(K)$,  and let $p>2$ for $G_n\ne D_n(K)$.
Set $\mathcal{P}=\{\mathcal{O}, \mathcal{L}, \mathcal{S}\}$.  An
indecomposable inductive system $\Phi$ is a BWM-system  if and
only if $\Phi=\otimes_{j=0}^s \Fr^j(\Phi^j)$, where $\Phi^j\in
\mathcal{P}$.  BWM-systems are finite unions of indecomposable
ones  and consist of modules with one dimensional weight spaces.
\end{theorem}

For $G_n=C_n(K)$ and $p=2$ the answer is more complicated, see Theorem~\ref{12new2}.

\section{Notation and preliminaries}

Let $\mathbb{Z}_{\ge0}$ be the set of nonnegative integers. For a simple algebraic group
$G$ over $K$ the symbol $\Lambda(G)$ denotes the set of weights of $G$, $R(G)$ is the  set of  roots of $G$; $\langle
\lambda,\alpha\rangle$ is the value of a weight
$\lambda\in\Lambda(G)$ on a root $\alpha\in R(G)$, and $\Irr G$ is
defined as for groups $G_n$. Throughout the text $\Lambda(M)$ is
the set of all weights of a $G$-module $M$. For a $G$-module $M$
denote by $v^+$ a nonzero highest weight vector of $M$ and by
$M^{\mu}$ the weight space in $M$ of a weight $\mu$. The
subspace of a linear space $L$ spanned by vectors $v_1,\ldots,
v_i$ is denoted by $\langle v_1,\ldots, v_i \rangle$, respectively. For
positive roots $\beta_1,\ldots,\beta_j$ denote by $G(\beta_1, \ldots,
\beta_j)$ the subgroup of $G$ generated by the root subgroups associated with $\pm\beta_1,\ldots,\pm\beta_j$.
In all cases where subgroups of
this form are considered, the roots $\beta_1,\ldots,\beta_j$ are
chosen such that they constitute a base of the root system of
$G(\beta_1,\ldots,\beta_j)$. In this situation the fundamental
weights of $G(\beta_1,\ldots,\beta_j)$ are determined with respect
to this base. If $H=G(\beta_1,\ldots, \beta_k)\subset G$ and
$\omega\in\Lambda(G)$, then $\omega{\downarrow}H$ is the
restriction of $\omega$ to $H$. For a $G$-module $M$ and a weight
vector $v\in M$ we denote the weight of $v$ with respect to a
subgroup $H\subset G$ by $\omega_H(v)$. Set
$\omega(v)=\omega_G(v)$.

In what follows $\varepsilon^n_i$ with $1\leq i\leq n+1$ for $G_n=A_n(K)$ and $1\leq i\leq n$ otherwise are weights of $V_n$, their labeling is standard and corresponds to \cite[Ch. VIII, \S 13]{Burb}. Put $G_n(i_1,\ldots,i_j)=G_n(\alpha_{i_1},\ldots,\alpha_{i_j})$.

We assume that $n>1$ in all cases where $n-1$ appears in formulas. For $k<n$ set $G_{n,k}=G_n(n-k+1,\ldots,n)$. As we have mentioned in the Introduction, $G_{n,k}\cong G_k$. Put $\Irr_k M=\Irr (M{\downarrow}G_{n,k})$.

\begin{theorem}[Jantzen~\cite{Jan1}, Smith~\cite{Smith}]
\label{Sm} Let $H=G_n(i_1,\dots,i_j)\subset G_n$. Then $K H v^+\subset
L(\omega)$ is an irreducible $H$-module with highest weight
$\omega_H(v^+)$ and a direct summand of the $H$-module $L(\omega)$.
\end{theorem}

Call $K H v^+$ in the previous theorem the {\em Smith factor} of $L(\omega)$ (with respect to $H$).

\begin{lemma}
\label{utv*} Let $M\in\Irr G_n$,  and let $\alpha$ be a long root
of $G_n$. Then $\delta(M)=\max_{\lambda\in\Lambda(M)} \langle
\lambda,\alpha \rangle$.
\end{lemma}

\begin{proof}
Denote by $\alpha_{\max}$ the maximal root in $R(G_n)$. As $\alpha_{\max}$ is a dominant
weight, $\langle\alpha_i,\alpha_{\max}\rangle\ge0$. This implies
$$
\delta(M)=\langle\omega(M),\alpha_{\max}\rangle=\max_{\lambda\in\Lambda(M)}\langle\lambda,
\alpha_{\max}\rangle.
$$
Since the
Weyl group acts transitively on the set of roots of the same length and $\alpha_{\max}$ is long,
$\max_{\lambda\in\Lambda(M)}\langle\lambda,\alpha\rangle=\max_{\lambda\in\Lambda(M)}\langle\lambda,\alpha_{\max}\rangle$
as required.
\end{proof}

\begin{corollary}
\label{cutv*}
In the assumptions of Lemma~$\ref{utv*}$ suppose that $\alpha$ is positive and set $H=G_n(\alpha)$. Then $\delta(M)=\max\{i\mid L(i\omega_1^1)\in \Irr(M{\downarrow}H)\}$.
\end{corollary}

\begin{proof}
Obviously,
\[
\max_{\lambda\in\Lambda(M)}\langle\lambda,\alpha\rangle=\max_{\mu\in\Lambda(M{\downarrow}H)} \langle\mu,\alpha\rangle=\max\{i\mid L(i\omega_1^1)\in \Irr(M{\downarrow}H)\}.
\]
It remains to apply Lemma~\ref{utv*}.
\end{proof}

\begin{corollary}
\label{ogr} Let $k<n$, $M\in\Irr G_n$, and $N\in\Irr_k M$. Assume that $k>1$ for $G_n=B_n(K)$.
Then $\delta(N)\le\delta(M)$.
\end{corollary}

\begin{proof}
Put
\[
\Lambda'=\{\lambda{\downarrow}G_k \mid \lambda\in \Lambda(M)\},
\]
$\beta=\alpha_{n-1}$ for $G_n=B_n(K)$ and $\beta=\alpha_n$ otherwise.
It is clear that $\Lambda(N)\subset \Lambda'$. By Lemma~\ref{utv*},
\[
\delta(N)=\max_{\lambda\in\Lambda(N)} \langle
\lambda,\beta \rangle\leq \max_{\lambda\in\Lambda'} \langle \lambda,\beta \rangle=\delta(M).
\]
\end{proof}

Recall the set of $A_n(K)$-modules $\mathcal{F}_n$ defined in~(\ref{Fn}).

\begin{lemma}
\label{o1} Let $G_n=A_n(K)$.
\begin{itemize}
\item[$(i)$]
For $1\leq i\leq n$ the set $\Irr_{n-1} L(\omega^n_i)=\{L(\omega^{n-1}_{i-1}),
L(\omega^{n-1}_i)\}$.
\item[$(ii)$]
Let $k<i\leq n-k+1$, $M\in\Irr G_n$, and
$\omega(M)=\omega_i^n$. Then $\Irr_k M=\mathcal{F}_k$.
\end{itemize}
\end{lemma}

\begin{proof}
$(i)$ Denote by $\wedge^i  V_n$ the $i$th wedge power of $V_n$. One has
$L(\omega^n_i)=\wedge^i  V_n$~\cite[Part~II, 2.15]{Jan}. Let $v_1,\ldots,v_{n+1}\in V_n$ and
$\omega(v_i)=\varepsilon^n_i$. Set $\Gamma= G_{n,n-1}$. One can assume that
$\varepsilon^n_1{\downarrow}\Gamma=0$ and $\Gamma$ fixes
$\langle v_2,\ldots, v_{n+1} \rangle$ and $v_1$. Then the
$\Gamma$-module $\langle v_2,\ldots, v_{n+1} \rangle$ is isomorphic
to $V_{n-1}$. Set
\[
U_1=\langle v_{k_1}\wedge\ldots\wedge v_{k_i}\mid 1< k_1<\ldots<k_i \leq n+1 \rangle
\]
and
\[
U_2=\langle v_1\wedge v_{l_1}\wedge\ldots\wedge v_{l_{i-1}} \mid 1< l_1<\ldots<l_{i-1} \leq n+1 \rangle.
\]
Then $\wedge^i V_n=U_1\oplus U_2$. One easily observes that
$\Gamma$ fixes $U_1$ and $U_2$, the $\Gamma$-module $U_1\cong
L(\omega^{n-1}_i)$ and $U_2\cong L(\omega^{n-1}_{i-1})$.

$(ii)$ Put $H_j=G_n(i-j+1,i-j+2,\ldots,i-j+k)$ for $1\leq j\leq k$ and
$H_0=G_n(1,\ldots,k)$. The subgroups $H_j$ are conjugate to $G_k$. Hence $\Irr(M{\downarrow}H_j)=\Irr(M{\downarrow}G_k)$.
By Theorem~\ref{Sm}, $L(\omega_j^k)\in\Irr(M{\downarrow}H_j)$ for $0\le j\le k$. Hence $\mathcal{F}_k\subset\Irr_k M$. It is well known that the maximal root $\alpha_{\max}=\alpha_1+\ldots+\alpha_n$ for $G_n=A_n(K)$. So $\delta(M)=1$. By Corollary~\ref{ogr}, $\delta(N)\leq \delta(M)$ for $N\in\Irr_k M$. Therefore $N\in \mathcal{F}_k$. This completes the proof.
\end{proof}

\begin{lemma}[{\cite[Proposition 1.4]{ZS}}]
\label{o2} Let $G_n=A_n(K)$, $H=G_n(1,\ldots,m,
m+2,\ldots,n)\subset G_n$, $0\leq c\leq p-1$, $0\leq i\leq n$. Then
\[
\begin{split}
& L(c\omega_i^n+(p-1-c)\omega_{i+1}^n){\downarrow}H=\\
& \quad =\oplus_{N(i,c)}L(c_1\omega_{i_1}^m+(p-1-c_1)\omega_{i_1+1}^m)\otimes
L(c_2\omega_{i_2}^{n-m-1}+(p-1-c_2)\omega_{i_2+1}^{n-m-1})
\end{split}
\]
with
\[
\begin{split}
N(i,c)=\{ (i_1,c_1), (i_2,c_2) \mid\quad  & 0\leq c_j<p, \quad 0\leq (p-1)(i_1+1)-c_1 \leq (p-1)(m+1), \\
& 0\leq (p-1)(i_2+1)-c_2\leq (p-1)(n-m),\\
& (p-1)(i_1+i_2+2)-c_1-c_2=(p-1)i+p-1-c \}.\\
\end{split}
\]
\end{lemma}

Here $H=H_1\times H_2$ with $H_1=G_n(1,\ldots,m)\cong A_m(K)$ and $H_2=G_n(m+2,\ldots,n)\cong A_{n-m-1}(K)$;
and the tensor product is the (external) product of 
 $H_1$- and $H_2$-modules.

Recall the set of $G_n$-modules $\mathcal{T}_n$ defined in (\ref{Tn}).

\begin{corollary}
\label{oc1} If $G_n=A_n(K)$, $k+1\leq i< n-k$, and
$\omega=c\omega_i^n+(p-1-c)\omega_{i+1}^n$, then $\Irr_k L(\omega)=\mathcal{T}_k$.
\end{corollary}

\begin{proof}
In Lemma~\ref{o2} take $m=k$ and observe that $H_1\cong G_k$. Now the corollary follows immediately from this lemma.
\end{proof}

\begin{corollary}
\label{csym} Let $G_n=A_n(K)$, $k<n$, and $\omega=a\omega_1^n$ with $0<a<p$. Then
$\Irr_k L(\omega)=\{L(b\omega_1^k)\mid 0\leq b\leq a\}$. \end{corollary}

\begin{proof} Argue as in the proof of Corollary~\ref{oc1} taking $m=k$, $i=0$, and $c=p-1-a$.
\end{proof}

\begin{lemma}[{\cite[Theorem, part C]{ZSVesti}}]
\label{rC} Let $p>2$, $n>1$, and
$G_n=C_n(K)$. Set
$M_1^n=L(\omega_{n-1}^n+\frac{p-3}{2}\omega_n^n)\in\Irr G_n$ and
$M_2^n=L(\frac{p-1}{2}\omega_n^n)\in\Irr G_n$.  Then
 $\Irr_{n-1} M_j^n=\{M_1^{n-1}, M_2^{n-1}\}$ for $j=1,2$.
\end{lemma}

\begin{lemma}
\label{l*} Let $n>2$ for $G_n=B_n(K)$ and $n>4$ for $G_n=D_n(K)$. Then $\Irr_{n-1} L(\omega_1^n)=\{L(0), L(\omega_1^{n-1}) \}$.
\end{lemma}

\begin{proof}
This is obvious and well known. We put some restrictions on $n$ to avoid complications connected with the isomorphisms between classical groups of small ranks from different series.
\end{proof}

The following lemma is also well known, but we fail to find an explicit reference.

\begin{lemma}
\label{lBD} If $G_n=B_n(K)$ and $n>2$ or $p=2$ and $G_n=C_n(K)$, then
$\Irr_{n-1} L(\omega_n^n)=\{L(\omega_{n-1}^{n-1})\}$. For $G_n=D_n(K)$ with $n>3$ one has
$\Irr_{n-1} L(\omega_n^n)=\Irr_{n-1} L(\omega_{n-1}^n)=\{L(\omega_{n-1}^{n-1}),
L(\omega_{n-2}^{n-1})\}$.
\end{lemma}

\begin{proof}
Let $M$ be one of the modules in question. If $G_n=B_n(K)$ or $D_n(K)$, it is well known that
$\omega(M)$ is a microweight and hence $\Lambda(M)$ coincides with
the orbit of  $\omega(M)$ under the action of the Weyl group.
Therefore $\Lambda(M)=\{(\pm\varepsilon^n_1+\ldots+\pm\varepsilon^n_n)/2\}$ with
all possible combinations of the ``plus'' and ``minus'' signs for
$G_n=B_n(K)$. If $G_n=D_n(K)$, then $\Lambda(M)$ consists of all
such weights with an odd or even number of the ``minus'' signs for
$M=L(\omega_{n-1}^n)$ or $L(\omega_n^n)$, respectively.

Let $p=2$ and $G_n=C_n(K)$. It is well known that in this case $\Lambda(M)$ is such as for $B_n(K)$. Indeed, using a special isogeny from $C_n(K)$ to $B_n(K)$, one easily concludes that $\dim M=2^n$ (as for the relevant $B_n(K)$-module), see~\cite[Subsection 5.3 and Theorem 5.4]{Humph}. Hence again $\Lambda(M)$ coincides with the orbit of $\omega(M)$.

The following arguments concern all the groups considered in this lemma. Let $M_+\subset M$ ($M_-\subset M$) be the sum of all weight subspaces $M^{\lambda}$ with $\lambda=\varepsilon^n_1/2+\mu$ ($\lambda=-\varepsilon^n_1/2+\mu$, respectively) where $\mu$ is a linear combination of the weights $\varepsilon^n_2, \ldots,
\varepsilon^n_n$. For $2\leq i\leq
n$ one can identify the restriction of the weight $\varepsilon^n_i$
to $G_{n-1}$ with the weight $\varepsilon_{i-1}^{n-1}\in\Lambda(G_{n-1})$. Taking into account
that for $2\leq i\leq n$ the roots $\alpha_i$ are linear
combinations of the weights $\varepsilon^n_i$ with $2\leq i\leq n$,
one can observe  that $G_{n,n-1}$ fixes $M_+$ and $M_-$. Analyzing
the weight structure of these $G_{n,n-1}$-modules, we conclude that
they are irreducible and have desired highest weights. This proves
the lemma.
\end{proof}

\begin{corollary}
\label{cBD}
Let $p=2$, $n>2$, and $G_n=C_n(K)$. Then
\[
\Irr_{n-1} L(\omega_1^n+\omega_n^n)=\{L(\omega_1^{n-1}+\omega_{n-1}^{n-1}),L(\omega_{n-1}^{n-1})\}.
\]
\end{corollary}

\begin{proof}
By~\cite[the corollary of Theorem 41]{Stein1}, for $G_k=C_k(K)$ and $k>1$ the $G_k$-module $L(\omega_1^k+\omega_k^k)\cong L(\omega_1^k)\otimes L(\omega_k^k)$. It is well known that $L(\omega_1^n){\downarrow}G_{n,n-1}$ is the direct sum of $L(\omega_1^{n-1})$ and two copies of $L(0)$. It has been shown in the proof of Lemma~\ref{lBD} that $L(\omega_n^n){\downarrow}G_{n,n-1}\cong L(\omega_{n-1}^{n-1})\oplus L(\omega_{n-1}^{n-1})$. This yields the corollary.
\end{proof}

\begin{proposition}
\label{deg*} Let $k<n$, $M\in\Irr G_n$, and $N\in\Irr_k M$.
Then $\wdeg N\leq\wdeg M$.
\end{proposition}

\begin{proof}
First assume that $k=n-1$. Put $\omega=\omega(M)$. For every
$\lambda\in\Lambda(M)$ one has $\lambda=\omega-\sum_{i=1}^n
b_i(\lambda)\alpha_i$ with $b_i(\lambda)\in\mathbb{Z}_{\ge0}$. For
$j\in\mathbb{Z}_{\ge0}$ put
\[
\Lambda_j=\{\lambda\in\Lambda(M) \mid b_1(\lambda)=j\}.
\]
It is obvious that $\Lambda_j\cap\Lambda_t=\varnothing$ for $j\ne t$
and
\[
\Lambda(M)=\Lambda_0 \cup \ldots \cup \Lambda_l
\]
for some $l$. Set
\[
U_j=\oplus_{\lambda\in\Lambda_j} M^{\lambda}.
\]
Then $U_j$ are $G_{n,n-1}$-modules and $M=U_0\oplus\ldots\ldots\oplus U_l$ as a
$G_{n,n-1}$-module. Hence $N$ is realized in a composition factor
of some module $U_s$. So $\wdeg N$ is not bigger then the maximal
weight multiplicity of the $G_{n,n-1}$-module $U_s$. It remains to
observe that the restrictions of distinct weights in $\Lambda_s$ to $G_{n,n-1}$ are distinct.
Indeed, assume $\mu,\nu\in\Lambda_s$ and $\nu\ne \mu$.
Obviously $b_1(\mu)= b_1(\nu)$. Hence $b_i(\mu)\ne b_i(\nu)$ for some $i$ with $2\leq i\leq
n$. This yields that $\mu{\downarrow}G_{n,n-1}\ne \nu{\downarrow}G_{n,n-1}$ and proves the
lemma for $k=n-1$. To complete the proof, it remains to apply induction on $n-k$.
\end{proof}

The following lemma is obvious.

\begin{lemma}
\label{degtenz}
Let $M_1$ and $M_2$ be $G_n$-modules. Then
$$\wdeg M_1^{[k_1]}\otimes M_2^{[k_2]}\geq \wdeg M_1 \cdot \wdeg M_2.$$
\end{lemma}

\section{Modules with small weight multiplicities for groups of type $A$}

In this section $G_n=A_n(K)$. For a module $M$  we assume
that $M^{\otimes 0}$ is the  trivial module. Recall the $\pdeg$ function defined in (\ref{pdegg}).

\begin{lemma}
\label{A1} \begin{itemize}
\item[(i)] Let $M\in\Irr G_n$ and $\pdeg M=d$. Then
$M\in\Irr  V_n^{\otimes d}$. If $N\in\Irr
 V_n^{\otimes d}$, then $\pdeg N\leq d$.

\item[(ii)] Let $M\in\Irr G_n$ and $\pdeg M^*=d$. Then
$M\in\Irr ( V_n^*)^{\otimes d}$. If $N\in\Irr
( V_n^*)^{\otimes d}$, then $\pdeg N^* \leq d$.
\end{itemize}
\end{lemma}

\begin{proof}
(i) By \cite[Subsection 5.2]{Green}, $V_n^{\otimes d}$ has a submodule isomorphic to the Weyl module with highest weight $\omega(M)$. This yields the first claim of (i).

Recall that $\omega_i^n=\varepsilon^n_1+\ldots+\varepsilon^n_i$,
$\alpha_i=\varepsilon^n_i-\varepsilon^n_{i+1}$ for $1\leq i\leq n$,
and $\varepsilon^n_1+\ldots+\varepsilon^n_{n+1}=0$. This implies that
if $\pdeg N=k$ and $\omega(N)=\sum_{i=1}^n b_i\varepsilon^n_i$, then
$\sum_{i=1}^n b_i=k$. It is clear that each weight $\mu\in\Lambda(V_n^{\otimes d})$ has the form $d\omega^n_1-\sum_{i=1}^n
c_i\alpha_i$ with $c_i\in\mathbb{Z}_{\ge0}$. This yields that $\pdeg
N\leq d$ for $N\in\Irr  V_n^{\otimes d}$ and completes the proof
of (i).

(ii) Take into account that $( V_n^*)^{\otimes d}\cong
( V_n^{\otimes d})^*$.
\end{proof}

Recall the sets $\mathcal{L}_n^d$ and $\mathcal{R}_n^d$ defined in (\ref{Ldn}).

\begin{proposition}
\label{LR} $\mathcal{L}_n^d=\{ M\in \Irr G_n \mid \pdeg M\leq
d \}$, $\mathcal{R}_n^d=\{ M\in \Irr G_n \mid \pdeg M^* \leq d
\}$.
\end{proposition}
\begin{proof}
This follows immediately from Lemma~\ref{A1}.
\end{proof}

\begin{proposition}
\label{deg} Let $M\in \Irr^p G_n$ and $\omega(M)\notin
\Omega_p(G_n)$. Assume that $\pdeg M\leq n$. Then $\wdeg M\geq
\pdeg M-2$.
\end{proposition}

\begin{proof}
Put $d=\pdeg M$ and $H=G_n(1,\ldots,d-1)$. Then $H\cong SL_d(K)$.
Note that $d>1$ as $\omega(M)\notin \Omega_p(G_n)$. Let $\omega(M)=\sum_{i=1}^n a_i\omega^n_i$. Since
$\omega(M)$ is not  fundamental, one easily observes that $a_j=0$ for $j>d-1$. Denote
by $N$ the Smith factor of $M$ associated with $H$ (see \ref{Sm}). It is clear that
$\pdeg M=\pdeg N=d$.

Now we can apply the Schur functor to the $H$-module $N$. Let $M(d,d)$ be the
category of the polynomial $GL_d(K)$-modules over $K$ which are
homogeneous of degree $d$, $\Sigma_d$ be the symmetric group of
degree $d$, and let $K\Sigma_d -\mod$ be the category of
$K\Sigma_d$-modules.
The Schur functor
\[
\mathcal{S}_d:M(d,d)\to K\Sigma_d -\mod
\]
sends a module $V\in M(d,d)$ to $V^0$ where $V^0$ is the
$(1,\dots,1)$-weight subspace in $V$ \cite[Chapter
6]{Green}. Alternatively, one can regard  $V$ as an $SL_d(K)$-module and define
$\mathcal{S}_d(V)$ as  the $0$-weight subspace of  $V$.

Let $\lambda=b_1\varepsilon^{d-1}_1+\ldots+b_d\varepsilon^{d-1}_d$ be the highest
weight of $N$. Note
that $b_1\geq\dots\geq b_d\geq 0$ and $b_1+\dots+b_d=d$. Hence
$\lambda=(b_1,\ldots,b_d)$ is a partition of $d$. The functor $\mathcal{S}_d$ is exact
and by \cite[6.4]{Green},
\[
\mathcal{S}_d(N)\cong D^{\lambda'}\otimes {\rm sgn}
\]
where $D^{\lambda'}$ is the irreducible $\Sigma_d$-module
corresponding to the partition $\lambda'$ dual to $\lambda$,
and $\sgn$ is the sign module for  $\Sigma_d$. Hence by Proposition~\ref{deg*},
$\wdeg M\geq\wdeg N\geq \dim D^{\lambda'}$.  If $\wdeg N<d-2$, then \cite{James} implies that $D^{\lambda'}\otimes {\rm sgn}$ is equal
to the trivial module  or $\sgn$. So $D^{\lambda'}$ is the trivial module  or $\sgn$ in
this case.  If $D^{\lambda'}$ is trivial,  then
its diagram is the row of $d$ boxes, therefore the diagram for $\lambda$ is the column of $d$ boxes and $N$ and $M$ are fundamental
modules  (recall that their highest weights are determined by the same formula).
By~\cite[Section~5, Example]{Klesch}, if $d=k(p-1)+r$ with $0\leq r<p-1$, then the  diagram for
$\sgn$ consists of $r$ rows of length $k+1$ and $p-1-r$ rows of
length $k$. In this case $\lambda$  has the diagram of $k$ rows of length $p-1$ and 1 row
of length $r$ and so $\omega(N)=(p-1-r)\omega^{d-1}_k+r\omega^{d-1}_{k+1}$
which implies  that $N$ and $M$ are truncated symmetric powers of the natural
modules. In both cases  $\omega(M)\in\Omega_p(G_n)$ which yields a contradiction. Hence
$\wdeg M\geq\wdeg N\geq d -2$.
\end{proof}

\begin{lemma}
\label{wdegVd}
Let $n\geq d$. Then $\wdeg V_n^{\otimes d}=d!$.
\end{lemma}

\begin{proof} Set  $T=V_n^{\otimes d}$.
Note that each weight $\lambda$ of $T$ is of the shape
$\lambda=b_1\varepsilon^n_1+\dots+b_d\varepsilon^n_d$ where
$(b_1,\ldots,b_d)$ runs over all $b_i\ge 0$ with $b_1+\dots+b_d=d$
and $\dim T^\lambda=\frac{d!}{b_1! b_2! \ldots b_d!}\le d!$. On
the other hand, for $\lambda=\varepsilon^n_1+\dots+\varepsilon^n_d$,
this dimension is exactly $d!$. Therefore, $\wdeg V_n^{\otimes
d}=d!$.
\end{proof}

Recall the $G_n$-modules 
$M_{n,L}(a_1,\ldots,a_d)=L(a_1\omega_1^n+\ldots+a_d\omega_d^n)$
and
$M_{n,R}(a_1,\ldots,a_d)=L(a_d\omega_{n-d+1}^n+\ldots+a_1\omega_n^n)$ 
($n\geq d$) defined in the Introduction.

\begin{lemma}
\label{Hn}
Let $n\geq d$ and $M_n=M_{n,L}(a_1,\ldots,a_d)$ or $M_{n,R}(a_1,\ldots,a_d)$. Set $H_{n,L}=G_{n+1}(1,\ldots,n)$ and $H_{n,R}=G_{n+1,n}$. Then $M_n$ is isomorphic to the Smith factor of $M_{n+1}$ with respect to the subgroup $H_{n,L}$ or $H_{n,R}$ for $M_n=M_{n,L}(a_1,\ldots,a_d)$ or $M_{n,R}(a_1,\ldots,a_d)$, respectively. In particular, $M_n\in \Irr_n M_{n+1}$.
\end{lemma}

\begin{proof}
This follows directly from Theorem~\ref{Sm}.
\end{proof}

\begin{proposition}
\label{1.1} Let $n\geq d$ and
$M=L(a_1\omega_1^n+\ldots+a_d\omega_d^n)\in\mathcal{L}_n^d$ or
$M=L(a_d\omega_{n-d+1}^n+\ldots+a_1\omega_n^n)\in\mathcal{R}_n^d$.
Then $\wdeg M\leq d!$. Moreover, $\wdeg M$ is determined by the sequence $(a_1,\ldots,a_d)$ and
does not depend on $n$.
\end{proposition}

\begin{proof}
Let $M\in\mathcal{L}_n^d$. We have $\pdeg M=j\leq d$ by
Lemma~\ref{A1}(i). Set $T= V_n^{\otimes j}$. Observe that
$M\in\Irr T$ by the same
lemma. Therefore $\wdeg M\le j!\le d!$ by Lemma \ref{wdegVd}.

Let $\lambda\in\Lambda(M)$ be dominant. As
$\lambda\in\Lambda(T)$, we have
$\lambda=b_1\varepsilon^n_1+\ldots+b_j\varepsilon^n_j$ with
$b_1\geq\ldots\geq b_j\geq 0$, $b_i\in\mathbb{Z}_{\ge0}$,  and
$b_1+\ldots+b_j=j$. Set $\omega=\omega(M)$. Then
$\lambda=j\omega_1^n-\sum_{i=1}^{j-1}
c_i\alpha_i=\omega-\sum_{i=1}^{j-1} d_i\alpha_i$ with $c_i$,
$d_i\in\mathbb{Z}_{\ge0}$. Denote by  $M_S$ the Smith factor of  $M$
associated with the subgroup $G_n(1,\ldots,j-1)\cong G_{j-1}$. By
Theorem~\ref{Sm},  $\dim M^{\lambda}=\dim M_S^{\lambda_S}$ for the
weight $\lambda_S=\lambda{\downarrow}G_n(1,\ldots,j-1)$. Since each weight in
$\Lambda(M)$ lies in the same orbit with a dominant weight under
the action of the Weyl group, we conclude that $\wdeg M=\wdeg M_S$
and hence does not depend on $n$.
To handle the case $M\in\mathcal{R}_n^d$,  consider $M^*$.
\end{proof}

\begin{lemma}
\label{Kle} Let  $1\leq j<k\leq n$, and let $\omega=\sum_{s=j}^k
a_s\omega_s^n$ be a dominant $p$-restricted weight of $G_n$ with
both $a_j$ and $a_k\ne 0$. Then
\[
\wdeg L(\omega)\geq k-j.
\]
\end{lemma}

\begin{proof} Write
$\omega=a_j\omega^n_j+a_{i_1}\omega^n_{i_1}+\ldots+a_{i_t}\omega^n_{i_t}+a_k\omega^n_k$
with $j<i_1<\ldots<i_t<k$ and $a_{i_1},\ldots, a_{i_t}\neq0$ ($t$
can be zero).  By~\cite[Proposition~1.21]{KleschA}, $\wdeg
L(\omega)\geq f(j,i_1, \ldots,i_t,k)$, where for $l$-tuples
$(u_1,\ldots, u_l)$ with $u_1<\ldots<u_l$ the integers
$f(u_1,\ldots, u_l)$ are determined  by the following recurrent
relations:
\begin{eqnarray*}
 f(u_1)&=& 1;\\
 f(u_1,u_2) &=& u_2-u_1;\\
 f(u_1,u_2,\ldots,u_l) &=& (u_2-u_1)f(u_2,\ldots,u_l)+ f(u_3,\ldots,
u_l) \quad \mathrm{for} \ l>2.\\
\end{eqnarray*}
We claim that  $f(j,i_1, \ldots,i_t,k)\geq k-j$. For $t=0$ this
holds by definition. Then apply induction on $t$. Let $t>0$. One
easily concludes that $f(u_1,\ldots,u_l)\geq1$ for all positive
integers $u_1,\ldots,u_l$. Now the induction hypothesis yields
that
$$
f(j,i_1, \ldots,i_t,k)=(i_1-j)f(i_1, \ldots,i_t,k)+f(i_2,
\ldots,i_t,k)\geq(i_1-j)(k-i_1)+1.
$$
(For $t=1$ we have $f(j, i_1,
k)=(i_1-j)(k-i_1)+1$.)  Note that $ab\geq a+b$ for $a$ and
$b\in\mathbb{N}$ and $a$, $b>1$. Hence $ab+1\geq a+b$ for all $a$
and $b\in\mathbb{N}$. This yields our claim and completes the
proof.
\end{proof}

Propositions~\ref{deg} and~\ref{1.1} imply that for groups of type
$A_n$ there exist classes of simple modules $M$ with
$\wdeg M$ arbitrary large,  but small with respect to $n$. Note that
for a generic simple $p$-restricted module $\wdeg M$
grows with the growth of $n$.

\begin{proposition}
\label{Cp} Let $M\in \Irr^p G_n$, $\omega(M)\notin \Omega_p
(G_n)$, and $n\geq16$. Assume $\pdeg M>n$ and $\pdeg M^*>n$. Then
$\wdeg M>\sqrt{n}/p -1$.
\end{proposition}

\begin{proof}
Let $\omega=\sum^j_{t=i}a_t\omega^n_t$ with $a_ia_j\neq0$,  $1\leq
i\leq j\leq n$. Due to Lemma~\ref{Kle} one can assume that
$j-i\leq \sqrt{n}/p -1$ (otherwise $\wdeg L(\omega)\geq j-i
>\sqrt{n}/p -1$ as required). Put $k=j-i+1$ and
$a=\sum^j_{t=i}a_t$.  Then $k\leq \sqrt{n}/p$ and
\begin{equation}
\label{eq*} a\leq k(p-1)< \sqrt{n}.
\end{equation}
Passing to $M^*$ if necessary, one can assume $i-1\leq
n-j$. For $1\leq i\leq s$ denote by $H_s$  the subgroup
$G_n(s,\ldots,n)\cong A_{n-s+1}(K)$.  So $H_1=G$ and the rank of $H_s$
is equal to $n-s+1>n/2$ for all $s\leq i$.

Let $L_s$ be the Smith factor of $L(\omega)$ with respect to $H_s$. Then
$\pdeg L_s=\pdeg L_i+(i-s)a$ for $1\leq s\leq i$. Note that
$$\pdeg L_i\leq ka\leq k^2(p-1)\leq
n(p-1)/p^2< n/2$$
since $p\geq 2$.

Fix minimal $s$ such that $\pdeg L_s\leq n/2$. Since  $\pdeg
L_1=\pdeg L(\omega)>n$, we have $s>1$. Then $\pdeg L_{s-1}=\pdeg
L_s+a>n/2$, so $\pdeg L_s >n/2-a$. Applying~(\ref{eq*}), we get
$n/2-a>n/2-\sqrt{n}$. As  the rank of $H_s$ is greater than $n/2$,
by Proposition~\ref{deg},
$$
\wdeg L(\omega)\geq \wdeg L_s\geq \pdeg L_s-2>
n/2-\sqrt{n}-2=\sqrt{n}(\sqrt{n}/2-1)-2\geq\sqrt{n}-2>\sqrt{n}/p-1
$$
since $n\geq16$ and $p\geq2$.
\end{proof}

Now we are ready to prove our first main result.

\begin{proof}[Proof of Theorem~$\ref{t1A}$]
Part $(i)$ is proved in Proposition \ref{Cp}
and part $(ii)$ follows from Lemma~\ref{A1} and  Propositions \ref{deg} and \ref{1.1}.
\end{proof}

\begin{corollary}
\label{1.5} Let $M\cong\otimes^i_{k=0}M_k^{[k]}$. If at
least one of $M_k$ satisfies the assumptions of Proposition~$\ref{Cp}$,
then $\wdeg M>\sqrt{n}/p -1$.
\end{corollary}

\begin{proof}
This follows immediately from  Lemma~\ref{degtenz} and Proposition~\ref{Cp}.
\end{proof}

Now we pass to modules that are not $p$-restricted.

\begin{lemma}
\label{tau} Let $M\in\Irr G_n$, $M=N_1\otimes
N_2^{[s]}$, $N_1,N_2\in\Irr G_n$, and let
$\delta(N_1)<p^s$. Then for any weight $\lambda\in \Lambda(M)$
there exists a unique pair $(\mu,\nu)$ with $\mu\in\Lambda(N_1)$,
$\nu\in\Lambda(N_2^{[s]})$, and $\lambda=\mu+\nu$.
\end{lemma}

\begin{proof}
It is obvious that $\lambda=\mu+\nu$ for some $\mu$ and $\nu$. Put
$N'=N_2^{[s]}$. Suppose that $\mu+\nu=\mu'+\nu'$ with
$\mu'\in\Lambda(N_1)$, $\nu'\in\Lambda(N')$, and $\mu\ne \mu'$.
Then $\mu-\mu'=\nu'-\nu$. Acting by the Weyl group, one can assume
that $\mu-\mu'$ (and hence $\nu'-\nu$) is dominant. Denote by
$\alpha_m$ the maximal root of $G_n$.
Note that  $\nu=p^s\xi$ and $\nu'=p^s\xi'$ with $\xi$ and
$\xi'\in\Lambda(N_2)$. Therefore
$$
p^s\langle \xi'-\xi,\alpha_m \rangle=
\langle \nu'-\nu,\alpha_m\rangle=
\langle
\mu-\mu',\alpha_m \rangle \leq 2\delta(N_1)<2p^s.
$$
This implies that $\langle \xi'-\xi,\alpha_m \rangle=1$, i.e. $\xi'-\xi$
is a fundamental weight. However, this difference is a radical
weight (i.e. a linear combination of roots). This yields a
contradiction and proves the lemma.
\end{proof}

Now consider tensor products of certain special modules with
relatively small $\wdeg M$.

\begin{theorem}
\label{CnonpA} Let $d\in\mathbb{N}$ and
\[
M=N_0\otimes\ldots\otimes N_l\in\Irr G_n.
\]
Assume that $\Omega(M)\notin\Omega(G_n)$,
\begin{equation}
\label{Fr1} N_t=\otimes_{s=i_{t-1}+1}^{i_t} M_s^{[s]}
\end{equation}
with $i_{-1}=-1$, $i_0<i_1<\ldots<i_l$, and for each $t$, $0\leq t\leq l$, one of the
following holds: $M_s\in \mathcal{L}^d_n$ for $i_{t-1}+1\leq s\leq
i_t$, or $M_s\in\mathcal{R}^d_n$ for all these $s$, or
$\omega(N_t)\in\Omega(G_n)$. Let $\delta(N_f)<p^{i_f+1}$ for all
$N_f$ with $f<l$ ($i_f$ are such as in~$(\ref{Fr1})$). Suppose that $\{
u_1,\ldots, u_k\}$ be the set of all indices $t$ for which
$\omega(N_t)\not\in\Omega(G_n)$. Set $l_j=i_{u_j}-i_{u_j-1}-1$ for $1\leq j\leq k$ and $d_j=d(1+p+\ldots+p^{l_j})$. Assume that $n\geq \max_{1\leq j\leq k}(d_j)$.
Then $\wdeg M\leq\prod_{j=1}^k d_j!$.
\end{theorem}

\begin{proof}
For $1\leq j\leq k$ set $s_j=i_{u_j-1}+1$ and $N'_j=\otimes_{g=0}^{l_j} M_{{s_j}+g}^{[g]}$. We have
$N_{u_j}=(N'_j)^{[s_j]}$. Hence $\wdeg N_{u_j}=\wdeg N_j'$.
Apply induction on $l$. If $l=0$, it is clear that $k=1$, $s_1=0$, $l_1=i_0$, and $d_1=d(1+p+\ldots+p^{i_0})$.
Then Proposition~\ref{LR} implies that $M\in\mathcal{L}^{d_1}$
or $\mathcal{R}^{d_1}$. Hence our assertion follows from
Proposition~\ref{1.1}. Assume that $l>0$ and the assertion holds
for $l-1$. Set $M'=N_0\otimes \ldots \otimes N_{l-1}$. Since
$\delta(N_j)<p^{i_j+1}$ for $j<l$, we get
$\delta(M')<p^{i_{l-1}+1}$. Then by Lemma~\ref{tau}, for each
$\lambda\in\Lambda(M)$ there exists a unique pair $(\mu, \nu)$ with
$\mu\in\Lambda(M')$,  $\nu\in\Lambda(N_l)$, and $\lambda=\mu+\nu$.
Then $\dim M^{\lambda}=\dim(M')^{\mu}\dim N_l^{\nu}$ and hence
$\wdeg M=\wdeg M'\wdeg N_l$. By the induction assumptions, $\wdeg
M'\leq\prod_{j=1}^{k-1} d_j!$ if $u_k=l$ and $\wdeg
M'\leq\prod_{j=1}^k d_j!$ otherwise. In the first case Proposition~\ref{LR} yields that $N'_l\in\mathcal{L}^{d_k}$
or $\mathcal{R}^{d_k}$. Hence $\wdeg N_l=\wdeg N'_l\leq d_k!$ by Proposition~\ref{1.1}.
In the second one $\omega(N_l)\in\Omega(G_n)$ and $\wdeg N_l=1$. This completes the proof.
\end{proof}

\noindent {\bf Remark}$\quad$ In some cases much stronger estimates
can be obtained. In particular, this holds if $n\geq d$,
$M=\otimes^f_{k=0}M_k^{[k]}$ with $M_k\in\Irr^p G_n$, and $\delta(M_k)<p$ for
all $k<f$. Then, applying Lemma~\ref{tau} and Proposition~\ref{1.1},  we can deduce that
$\wdeg M\leq (d!)^N$, where $N$ is the number of indices $k$ for
which $\omega(M_k)\not\in\Omega_p(G_n)$.

Proposition~\ref{tenz} shows that our
assumptions on $\delta(N_f)$ play a crucial role in
Theorem~\ref{CnonpA}.

\begin{proposition}
\label{tenz} Let $i,l\in \mathbb{N}$ with $i<l-1$ and $M,N\in\Irr
G_n$. Assume that $\omega(M)=\sum_{t=1}^ia_t\omega^n_t=\sum_{k=0}^jp^k\lambda_k$
with $p$-restricted $\lambda_k$
and $\omega(N)=\sum_{t=l}^nb_t\omega^n_t\neq0$ is $p$-restricted. Suppose
that $\delta(M)\ge p^{j+1}$. Set $Q=M\otimes N^{[j+1]}$. Then
$\wdeg Q\ge l-i-1$. The same holds if
$\omega(M)=\sum_{t=l}^nb_t\omega^n_t$, $\omega(N)=\sum_{t=1}^i
a_t\omega^n_t$,  and other assumptions of the proposition are valid. In
particular, in this situation $\wdeg Q\ge n-m-i$ if $M\in
\mathcal{L}^i_n$, $N\in\mathcal{R}^m_n$ or vice versa.
\end{proposition}

\begin{proof}
We will consider the case where $\omega(M)=\sum_{t=1}^i a_t\omega_t^n$ and
$\omega(N)=\sum_{t=l}^nb_t\omega^n_t\ne 0$. The proof for the other case is similar.

Taking maximal possible $l$, we can suppose that $b_l\ne 0$. Put
$c=\delta(M)$ and write down the $p$-adic expansion
$c=\sum_{k=0}^uc_kp^k$ with  $0\leq c_k<p$.

(a) First assume that $c_{j+1}\neq0$. Set
$\Gamma=G_n(\alpha_1+\ldots+\alpha_i, \alpha_{i+1}, \ldots,
\alpha_n)$. Observe that $\Gamma$ is conjugate to $G_{n-i+1}$, the
group $G_n(i+1,\ldots,n)$ is conjugate to $G_{n-i}$ and
$G_n(i+1,\ldots,l)$ is conjugate to $G_{l-i}$. We have
$\langle \omega(M),\alpha_1+\ldots+\alpha_i\rangle=c$.
Then one easily concludes that $L(c\omega_1^{n-i+1})\in\Irr(M{\downarrow}\Gamma)=\Irr_{n-i+1}M$. By the Steinberg
tensor product theorem (\ref{Fr}),
$L(c\omega_1^{n-i+1})=\otimes^u_{k=0}L(c_k\omega_1^{n-i+1})^{[k]}$. By Corollary~\ref{csym},
$L(0)\in\Irr_{n-i} L(c_k\omega_1^{n-i+1})$ for $0\leq k\leq u$ and
$L(\omega_1^{n-i})\in\Irr_{n-i} L(c_{j+1}\omega_1^{n-i+1})$.
Hence
\[
L(p^{j+1}\omega_1^{n-i})\in\Irr_{n-i}L(c\omega_1^{n-i+1})\subset\Irr_{n-i}M.
\]
So by Theorem~\ref{Sm},   $L(p^{j+1}\omega_1^{l-i})\in\Irr_{l-i}M$.
Applying Theorem~\ref{Sm} to the restriction $N{\downarrow}G_n(i+1,\ldots,l)$, we get
that $L(b_l\omega_{l-i}^{l-i})\in\Irr_{l-i}N$. Consequently,
$F=L(p^{j+1}(\omega_1^{l-i}+b_l\omega_{l-i}^{l-i}))\in\Irr_{l-i}Q$.
Lemma~\ref{Kle} and Proposition~\ref{deg*} imply that $\wdeg Q\geq l-i-1$.

(b) Now let $c_{j+1}=0$. Then $\sum^i_{k=1}a_k=c\geq p^{j+2}$. Fix minimal $s$ with
$\sum^s_{k=1}a_k>\sum_{k=0}^jc_kp^k$. Put $\Sigma_s=a_1+\ldots+a_s$ and $c_s=c-\Sigma_s$. Since all $a_k<p^{j+1}$,
we get $\Sigma_s<\sum_{k=0}^jc_kp^k+p^{j+1}$ and hence $s<i$.
Write $\Sigma_s=\sum^u_{k=0}d_k p^k$ and $c_s=\sum^u_{k=0}g_k p^k$ with $0\leq d_k<p$ and $0\leq g_k<p$.
One can observe that either $\Sigma_s=\sum^j_{k=0}d_k p^k$ or $\Sigma_s=\sum^j_{k=0}d_k p^k+p^{j+1}$ with $\sum^j_{k=0}d_k p^k<\sum^j_{k=0}c_k p^k$. So in both cases $g_{j+1}=p-1$.

Set $H=G_n(s+1,\ldots,n)$. Then $H\cong G_{n-s}$. Let $M_s$ be the Smith factor of $M$ with respect to
$H$. Then $c_s$ is the value of $\omega(M_s)$ on the maximal root of $H$. Now we can proceed as in Part (a)
using $H$, $M_s$, and the Smith factor of $N$ with respect to $H$ rather than $G_n$, $M$, and $N$.
\end{proof}

\section{The Steinberg tensor product theorem for inductive systems}
\label{steinberg}

In this section we 
study  arbitrary inductive systems of representations for the sequence~(\ref{seq1})
and prove an analogue of the Steinberg product theorem for such systems.

Let $\Phi=\{\Phi_n\mid n\in\mathbb{N}\}$ be an inductive system. Put $\delta(\Phi_n)=\{\max \delta(\omega) \mid L(\omega)\in\Phi_n\}$.

\begin{lemma}
\label{delta} Assume that $n\in\mathbb{N}$ and
$n>2$ for $G_n=B_n(K)$. Then for
an inductive system $\Phi$ one has
$\delta(\Phi_{n+1})=\delta(\Phi_n)$.
\end{lemma}

\begin{proof}
Fix any $L(\lambda)\in\Phi_n$ and $L(\mu)\in\Phi_{n+1}$ with
$\delta(\Phi_n)=\delta(\lambda)$ and
$\delta(\Phi_{n+1})=\delta(\mu)$. Put $H=G_{n+1}(n)$ for
$G_{n+1}=B_{n+1}(K)$ and $H=G_{n+1}(n+1)$ in the other cases. Hence $H\cong A_1(K)$. Recall that $G_n$ is identified with $G_{n+1,n}=G_{n+1}(2,\ldots,n+1)$.
So we can assume that $H\subset G_n$. Set
\[
I_l=\cup_{\varphi\in\Phi_l} \Irr (\varphi{\downarrow}H)
\]
for $l=n$ and $n+1$. It is clear that
\[
\Irr (\varphi{\downarrow}H)=\cup_{\psi\in\Irr(\varphi{\downarrow}G_n)}\Irr (\psi{\downarrow}H)
\]
for $\varphi\in\Phi_{n+1}$.
Now it follows from the definition of an inductive system that $I_n=I_{n+1}$. Corollary~\ref{cutv*} implies that $\delta(\mu)=\max \{i \mid L(i\omega^1_1)\in I_{n+1}\}$ and
$\delta(\lambda)=\max \{i \mid L(i\omega^1_1)\in I_n\}$. Hence  $\delta(\Phi_{n+1})=\delta(\Phi_n)$.
\end{proof}

Set  $\delta(\Phi)=\delta(\Phi_n)$ for $n>2$. Lemma~\ref{delta} shows that
$\delta(\Phi)$ is well defined.

For the groups of type $A$ the previous lemma was proven
in~\cite[Lemma 2.4]{BarSupr}. Note that for any dominant
weight $\omega=a_1\omega_1^n+\dots+a_n\omega_n^n$ of $A_n(K)$ one has
$\delta(\omega)=a_1+a_2+\dots+a_n$.

\begin{lemma}
\label{pten} Let $\Phi$ and $\Psi$ be inductive systems of
representations. Then $\Fr(\Phi)$ and $\Phi\otimes\Psi$ are
inductive systems of representations.
\end{lemma}

\begin{proof}
The claim on $\Fr (\Phi)$ follows immediately from the definition
of an inductive system since for $M\in\Irr G_{n+1}$
\[
\Irr_n (M^{[1]}) =\{ \mu^{[1]} \mid \mu\in\Irr_n M \}.
\]
Clearly,  the set $(\Phi\otimes\Psi)_n$ is finite.
It remains to note that restricting  representations to subgroups commutes
with taking tensor products.
\end{proof}

\begin{lemma}
\label{par3} Let $T\subset\mathbb{N}$ be infinite. Assume that $R_t\subset\Irr G_t$ is nonempty for each $t\in T$ and
that there exists $k\in\mathbb{N}$ such that $\delta(\varphi)<k$ for all
$\varphi\in R_t$ and  all $t$. Denote by $\Pi_n$ the set of all
$\pi\in\Irr G_n$ such that $\pi$ is a composition factor of the
restriction $\mu\downarrow G_n$ for some $t>n$, $t\in T$, and $\mu\in
R_t$. Suppose also that  $R_t\subset \Pi_t$ for all $t$. Then
$\Pi=\{\Pi_n\mid n\in \mathbb{N}\}$ is an inductive system of representations.
\end{lemma}

\begin{proof} Let $\rho\in\Pi_{n+1}$. The construction of $\Pi$ implies that there exist $t>n+1$ and $\psi\in R_t$ with $\rho\in\Irr_{n+1}\psi$. So if $\varphi\in\Irr_n\rho$, then $\varphi\in\Irr_n\psi$ and hence $\varphi\in\Pi_n$. On the other hand, for each  $\mu\in\Pi_n$ there exist $u>n$ and $\nu\in R_u$ with $\mu\in\Irr_n\nu$. If $u>n+1$, the set $\Irr_{n+1}\nu\subset\Pi_{n+1}$ and, obviously, $\mu\in\Irr_n\lambda$ for some    $\lambda\in\Irr_{n+1}\nu$. Since $R_{n+1}\subset\Pi_{n+1}$ by the assumptions of the lemma, for $u=n+1$ the representation $\mu\in\Irr_n\lambda$ for some    $\lambda\in\Pi_{n+1}$ as well. It remains to show that $\Pi_n$ is finite. As   $\Pi_1=\bigcup_{\rho\in\Pi_2}\Irr_1\rho$, we can assume that $n>1$.  It follows from Corollary~\ref{ogr} that $\delta(\varphi)\leq k$. It is clear that the number of inequivalent irreducible representations of $G_n$ with this property is finite.
\end{proof}

\begin{corollary}
\label{cgen}  
Lemma~$\ref{par3}$ holds if we replace
the condition that $\delta(\varphi)<k$ for all $\varphi\in R_t$ and  all $t$, 
by the condition that there exists an inductive system $\Phi$ with  $R_t\subset\Phi_t$ for all $t$. 
\end{corollary}

\begin{proof} Corollary~\ref{ogr} implies that $\delta(\pi)<\delta(\Phi)$ for all $\pi\in R_t$. So we can apply Lemma~\ref{par3}.
\end{proof}

\begin{definition}\label{Di}
Let $\Psi\subset\Phi$ be inductive systems of representations and
the embedding be  proper. Put $\Xi_n=\Phi_n\setminus \Psi_n$. Denote
by $D(\Phi,\Psi)$ the inductive system of representations generated
by $\Xi_n$ and call it the {\em difference} of two inductive
systems.
\end{definition}

It is shown in \cite[Section 4]{BSA2rev} that $D(\Phi,\Psi)$ is well defined. (We emphasize that though \cite{BSA2rev} is devoted to general linear and special linear groups, the arguments on the difference of induction systems at the beginning of Section 4 of that paper hold for inductive systems for the sequence~(\ref{seq1}) for all four series of the classical groups.)  Since the embedding is proper, for any $n\in\mathbb{N}$ there exists
$n_0>n$ such that the set $\Xi_{n_0}\ne\varnothing$. Hence
$D(\Phi,\Psi)_n\ne\varnothing$ for all $n$. One obviously has
$\Phi=\Psi\cup D(\Phi,\Psi)$.

\begin{lemma}
\label{index} Let $\Phi$ be an indecomposable inductive system. Then
for each two representations $\varphi\in \Phi_k$ and $\psi\in \Phi_l$ there
exist $m>\max\{k,l\}$ and $\xi\in \Phi_m$ such that $\varphi\in\Irr_k\xi$ and  $\psi\in\Irr_l\xi$.
\end{lemma}

\begin{proof}
Set $t=\max\{k,l\}$. For each $n>t$ put $P_n=\{\rho\in\Phi_n \mid \varphi\in\Irr_k
\rho\}$. It is clear that $P_n\ne\varnothing$ and for any $\mu\in
P_n$ there exists $\nu\in P_{n+1}$ such that $\mu\in\Irr_n \nu$.
Hence $\mathcal{P}=\langle P_n \mid n>t \rangle$ is an inductive
system by Corollary~\ref{cgen}. We claim that $\mathcal{P}=\Phi$. Indeed, otherwise
$D(\Phi,\mathcal{P})=\Phi$ as $\Phi$ is indecomposable. However,
$\varphi\notin\Irr_k \psi$ if $\psi\in\Phi_n\setminus P_n$  by the
construction of $P_n$. This yields a  contradiction as
$D(\Phi,\mathcal{P})$ is generated by the collection
$\Phi_n\setminus P_n$. Hence $\mathcal{P}=\Phi$. So there exists
$m>t$ such that $\psi\in\Irr_l \rho$ for $\rho\in P_m$.
\end{proof}

\begin{corollary}
\label{i} Let $\Phi$ be an indecomposable inductive system
and let $\varphi_1\in \Phi_{n_1},\ldots, \varphi_l\in \Phi_{n_l}$. Then there
exist $m>\max\{n_1,\ldots, n_l\}$ and $\xi\in \Phi_m$ such that $\varphi_j\in\Irr_{n_j}\xi$ for
$1\leq j\leq l$.
\end{corollary}

\begin{proof}
Use Lemma~\ref{index} and induction on $l$.
\end{proof}

\begin{proof}[Proof of Theorem~$\ref{anSt}$]
Since $\delta(\varphi)\leq \delta(\Phi)$ for all $n$ and all
$\varphi\in\Phi_n$, by the Steinberg tensor product theorem~(\ref{Fr})
there exists an integer $k=k(\Phi)$ such that $\varphi=\varphi_0 \otimes
\varphi_1^{[1]} \otimes \ldots \otimes \varphi_k^{[k]}$ with
$\varphi_j\in\Irr^p G_n$,  for all $n$ and all
$\varphi\in\Phi_n$. Fix minimal  such $k$. Then the representations $\varphi_j$, $0\le j\le k$,
are uniquely determined (some of them can be trivial). We will use this notation
until the end of the proof.

Set
\begin{eqnarray*}
 S_n&=&\{\varphi\in\Phi_n \mid \delta(\varphi)=\delta(\Phi) \}, \quad 
S=\cup_{n=1}^{\infty} S_n,\\
 S_n^0&=&\{ \varphi\in S_n \mid \delta(\varphi_0)=\max_{\psi\in S} \delta(\psi_0)\}, \quad  
S^0=\cup_{n=1}^{\infty} S_n^0,\\
 S_n^{0,\ldots, j}&=&\{ \varphi\in S^{0,\ldots, j-1}_n \mid \delta(\varphi_j)=
 \max_{\psi\in S^{0,\ldots, j-1}} \delta(\psi_j)\}, \quad
S^{0,\ldots, j}=\cup_{n=1}^{\infty} S_n^{0\ldots j}\\
\end{eqnarray*}
for $1\leq j\leq k-1$. Set $T_n=S_n^{0,1,\ldots, k-1}$ and $T_n^j=\{\varphi_j \mid \varphi\in T_n\}$ for $0\leq j\leq k$. The sets $T_n^j$ will be used to generate tensor factors for $\Phi$.

Since $\delta(\varphi)\leq \delta(\Phi)$ for all $\varphi\in\Phi_l$ and all $l$, it is clear that $T_n$ is well defined and  $T_n\ne \varnothing$ for some $n$. Choose minimal $n$ with this  property and denote it by $n_{\min}$. Now we shall prove the following  claim: if
$m>n\geq n_{\min}$,  $\varphi\in T_n$,  $\psi\in\Phi_m$,  and
$\varphi\in \Irr_n \psi$, then
\begin{equation}
\label{rescomp}
\delta(\varphi_j)=\delta(\psi_j),  \varphi_j\in\Irr_n \psi_j
\end{equation}
for $0\leq j\leq k$.  Hence such $\psi\in T_m$.

Fix  $\psi\in\Phi_m$ with $\varphi\in\Irr_n\psi$ (such $\psi$ do exists as $\Phi$ is an inductive system). Since restricting to subgroups commutes with the morphism $\Fr$ and taking tensor products, one can observe that
$$
\Irr_n\psi=\cup_{(\tau^0,\ldots,\tau^k)} \Irr (\otimes_{j=0}^k (\tau^j)^{[j]}),
$$
where the union is taken over all tuples
$(\tau^0,\ldots,\tau^k)$ with $\tau^j\in\Irr_n \psi_j$. Fix a tuple
$(\tau^0,\ldots,\tau^k)$ that yields $\varphi$ and set $\tau=\otimes_{j=0}^k (\tau^j)^{[j]}$. In fact, we shall show that all $\tau^j\in\Irr^p G_n$ and so $\tau^j=\tau_j$ for $0\leq j\leq k$, but this requires some explanations.  One  has $\tau=\tau^0_0\otimes\rho^{[1]}$, where $\rho$ is a representation of $G_n$ (not necessarily irreducible). The Steinberg tensor product theorem implies that each representation in $\Irr \tau$ has the form $\tau^0_0\otimes\lambda^{[1]}$ with $\lambda\in\Irr G_n$. Hence $\varphi_0=\tau^0_0$. Similar arguments yield that if $0<l\leq k$ and  $\tau^0,\ldots,\tau^{l-1}\in\Irr^p G_n$, then   $\Irr \tau$ consists of representations of the form
$(\otimes^{l-1}_{j=0}(\tau^j)^{[j]})\otimes(\tau_0^l)^{[l]}\otimes\mu^{[l+1]}$ with $\mu\in\Irr G_n$ and therefore in this case
\begin{equation}
\label{eql}
\tau^j=\varphi_j  \mbox{ for } 0\leq j\leq l-1, \quad \varphi_l=\tau_0^l.
\end{equation}
Obviously, we have $\delta(\rho)=\sum^k_{j=0}p^j\delta(\rho_j)$ for each $\rho\in\Irr G_l$ and all $l$. By Corollary~\ref{ogr}, $\delta(\varphi)\leq\delta(\psi)$ and $\delta(\tau^j)\leq\delta(\psi_j)$ for $0\leq j\leq k$.  This implies that $\delta(\varphi_0)\leq \delta(\psi_0)$ and $\psi\in S_m$ as $\varphi\in S_n$. Now we start proving (\ref{rescomp}) using  the induction on $j$. At each step we shall also show that $\tau^j\in\Irr^p G_n$. Since $\varphi\in S_n^0$ and $\varphi_0=\tau^0_0$, we conclude that $\delta(\varphi_0)=\delta(\tau^0)=\delta(\psi_0)$  and $\tau^0\in\Irr^p G_n$. So $\varphi_0\in\Irr_n \psi_0$ and (\ref{rescomp}) holds for $j=0$. It is clear that
$\psi\in S_m^0$. Now let $0<j<k$ and assume that for $0\leq l<j$ Formula~(\ref{rescomp}) holds and
$\tau^l\in\Irr^p G_n$. The construction of the sets $S^{0,\ldots ,t}$ yields  that $\psi\in S^{0,\ldots, j-1}$. By (\ref{eql}), $\varphi_j=\tau_0^j$. As $\varphi\in S_n^{0,\ldots, j}$ and
$\delta(\varphi_j)\leq\delta(\tau^j)\leq\delta((\psi_j)$, we can deduce that $\delta(\varphi_j)=\delta(\tau^j)=\delta(\psi_j)$  and $\tau^j\in\Irr^p G_n$. So $\varphi_j\in\Irr_n\psi_j$ and (\ref{rescomp}) holds for $j$. Finally, suppose that (\ref{rescomp}) is valid and $\tau^j\in\Irr^p G_n$ for $0\leq j<k$. The choice of $k$ shows that $\tau^k\in\Irr^p G_n$. Then $\tau^k=\varphi_k$ by (\ref{eql}) and hence $\varphi_k\in\Irr_n\psi_k$. Naturally, $\delta(\varphi_k)=\delta(\psi_k)$ since $\varphi$ and $\psi\in S$ and $\delta(\varphi_j)=\delta(\psi_j)$ for $0\leq j<k$. This completes the proof of the claim.

Now it is clear that $T_n$ and hence all $T_n^j\neq\varnothing$ for   $n\geq n_{\min}$. Let $\mu\in T_n^j$ with
$0\leq j\leq k$. Then $\mu=\rho_j$ for some $\rho\in T_n$.   We have shown above that  there exists
$\lambda\in T_{n+1}$ with $\rho\in\Irr_n \lambda$ and $\rho_j\in\Irr_n\lambda_j$. Naturally,
$\lambda_j\in T_{n+1}^j$. It is clear that $\delta(\mu)\leq \delta(\varphi)$. Now Lemma~\ref{par3} yields that  the collections $\Theta^j=\langle T_n^j\rangle$ are inductive systems. Put
$\Theta=\otimes_{j=0}^k \Fr^j(\Theta^j)$ and prove that
$\Phi=\Theta$. As $\Phi$ is indecomposable, Lemma~\ref{index}
implies that for every $\varphi\in\Phi_n$ and $\psi\in T_k$ with $k\geq n_{\min}$ there exists $m>\max\{n,k\}$ and
$\rho\in \Phi_m$ with $\varphi\in\Irr_n\rho$ and $\psi\in\Irr_k\rho$. It follows from Formula~(\ref{rescomp}) and the phrase just below this formula that $\rho\in T_m$. Hence the  construction of $\Theta$ yields  that $\Phi\subset\Theta$. By the definition of a tensor  product of inductive systems, now it suffices to prove the following: if $\rho=\otimes_{j=0}^k \rho_j^{[j]}$ with $\rho_j\in\Theta^j_n$, then $\Irr\rho\subset\Phi_n$. The construction of the systems $\Theta^j$ implies that there exist $m>n$ and representations
$\theta_j\in T_m^j$ with $\rho_j\in\Irr_n \theta_j$. Set
$\theta=\otimes_{j=0}^k \theta_j^{[j]}$. As $\Phi$ is an inductive
system, now it remains to show that $\theta\in\Phi_m$. By the
definition of $T_m^j$, there exist representations $\psi^j\in T_m$
with $\theta_j=\psi_j^j$. Since $\Phi$ is indecomposable,
Corollary~\ref{i} implies that for some $l>m$ there exists
$\zeta\in \Phi_l$ with $\psi^j\in \Irr_m\zeta$. By Formula~(\ref{rescomp}),   $\psi^j_j\in \Irr_m\zeta_j$ for
$0\leq j\leq k$. Hence $\theta\in\Irr_m\zeta\subset\Phi_m$ as desired.
\end{proof}

To describe BWM-systems, we also need the following lemma on tensor products of inductive systems that are generated by collections $R_n$ that consist of a single $p$-restricted representation of $G_n$.

\begin{lemma}
\label{LL} Let $j\in\mathbb{N}$ and  $M_{nt}\in\Irr^p G_n$ for $0\leq t\leq j$ and $n\geq d_t$. Assume that $\delta(M_{nt})\leq c$ for some constant $c$ and $M_{nt}\in\Irr_n M_{n+1,t}$ for $0\leq j\leq t$ and $n\geq d_t$. Set   $d=\max\{d_t\mid 0\leq t\leq j\}$ and $M_n=\otimes^j_{t=0}M_{nt}^{[t]}$ for $n\geq d$. Then
$\langle M_{nt}\mid n\geq d_t\rangle$ and $\langle M_n\mid n\geq d\rangle$ are inductive systems and
$$
\langle M_n\mid n\geq d\rangle=\otimes^j_{t=0}\Fr^t\langle M_{nt}\mid n\geq d_t\rangle.
$$

\end{lemma}

\begin{proof}
Set $\mathcal{M}=\langle M_n\mid n\geq d\rangle$ and $\mathcal{M}^t=\langle M_{nt}\mid n\geq d_t\rangle$. By the Steinberg tensor product theorem (\ref{Fr}), the modules $M_n$ are irreducible. Observe that $\delta(M_n)\leq c\sum_{t=0}^j p^t$ and $M_n\in\Irr_n M_{n+1}$. Now Lemma~\ref{par3} implies that $\mathcal{M}$ and $\mathcal{M}^t$ are inductive systems. Put $\mathcal{P}=\otimes^j_{t=0}\Fr^t\langle M_{nt}\mid n\geq d_t\rangle$. As $M_n\in\mathcal{P}_n$ and $\mathcal{P}$ is an inductive system, $\mathcal{M}\subset \mathcal{P}$.
Taking into account the definition of a tensor product of inductive systems, it remains to
prove that for each collection $(N_0,\ldots, N_j)$ with $N_t\in \mathcal{M}_n^t$ the set $S=\Irr(\otimes_{t=0}^j N_t^{[t]})\subset \mathcal{M}_n$. As $M_{nt}\in\Irr_n M_{n+1,t}$ and the sets $\mathcal{M}_n^t$ are finite,
the construction of the systems $\mathcal{M}^t$ implies that for $q$ large
enough $\mathcal{M}^t_n\subset\Irr_n M_{qt}$ for all $t$, $0\leq t\leq j$.
Hence $S\subset \Irr_n M_q\subset \mathcal{M}_n$. This completes the proof.
\end{proof}

\section{Inductive systems with bounded weight multiplicities for special linear groups}
\label{inductive}

In this section we classify the BWM-systems for $G_n=A_n(K)$.
We will denote by  $\mathbb{N}_j$ the set of integers $s$ with $ 0 \leq s \leq j$.

Recall the collections $\mathcal{L}^l$, $\mathcal{R}^l$,
$\mathcal{F}$, and $\mathcal{T}$ defined in the Introduction.

\begin{lemma}
\label{ind} The collections $\mathcal{L}^l$, $\mathcal{R}^l$
($l\in\mathbb{N}$), $\mathcal{F}$, and $\mathcal{T}$ are inductive
systems of representations for the groups  $A_n(K)$.
\end{lemma}

\begin{proof}
By Lemma~\ref{o1}(i), $\Irr_{n-1}
 V_n=\{L(0),  V_{n-1}\}$. Hence
$\Irr  V_n^{\otimes l}\subset \Irr_n  V_{n+1}^{\otimes l}$,
$\Irr_{n-1}  V_n^{\otimes
l}\subset \cup_{j\leq l} \Irr  V_{n-1}^{\otimes j}$
and
$\Irr_{n-1}\varphi\in \mathcal{L}^l_{n-1}$ for any $\varphi\in
\mathcal{L}^l_n$. Consequently, $\mathcal{L}^l$ is an inductive
system. The proof for $\mathcal{R}^l$ is similar.

For $\mathcal{F}$ and $\mathcal{T}$ the lemma follows from Lemma~\ref{o1} and
Lemma~\ref{o2}, respectively. This completes the proof.
\end{proof}

Recall the $G_n$-modules 
$M_{n,L}(a_1,\ldots,a_d)=L(a_1\omega_1^n+\ldots+a_d\omega_d^n)$
and
$M_{n,R}(a_1,\ldots,a_d)=L(a_d\omega_{n-d+1}^n+\ldots+a_1\omega_n^n)$ 
($n\geq d$) defined in the Introduction.

\begin{lemma}
\label{lLR} The systems
$$
C_L(a_1,\ldots,a_d)=\langle M_{n,L}(a_1,\ldots,a_d)\mid n\geq d\rangle
$$
and
$$
C_R(a_1,\ldots,a_d)=\langle M_{n,R}(a_1,\ldots,a_d)\mid n\geq d\rangle
$$
are well defined.
\end{lemma}

\begin{proof}
For $n\geq d$ set $M_n=M_{n,L}(a_1,\ldots,a_d)$ or $M_{n,R}(a_1,\ldots,a_d)$ (the index ``$L$'' or ``$R$'' is the same for all $n$). Obviously, $\delta(M_n)=a_1+\ldots+a_d$. By Lemma~\ref{Hn}, $M_n\in\Irr_n M_{n+1}$. It remains to apply Lemma~\ref{par3}.
\end{proof}

\begin{proposition}
\label{TF} Assume that $S_1\cup S_2\cup S_3=\mathbb{N}_j$,
$S_i\cap S_k=\varnothing$ for $i\ne k$, $S_3=\varnothing$ if
$p=2$, and $S_2\cup S_3\ne\varnothing$. If $S_1\ne\varnothing$,
for each $k\in S_1$ set $M_{n,k}=M_{n,L}(a_{1k},\ldots,a_{dk})$ or
$M_{n,R}(a_{1k},\ldots,a_{dk})$, where $0\leq
a_{1k},\ldots,a_{dk}< p$ and the index ``$L$'' or ``$R$'' and the sequence $a_{1k},\ldots,a_{dk}$ are the same
for all $n\geq d$. Put $\Psi^k=\langle M_{n,k} \mid n\geq d\rangle$ for
$k\in S_1$, $\Psi^k=\mathcal{F}$ for $k\in S_2$,
$\Psi^k=\mathcal{T}$ for $k\in S_3$, and $\Psi=\otimes_{k=0}^j
\Fr^k(\Psi^k)$. Let $\Phi$ be an inductive system. Assume that for
each $l$ there exist  $n$ and a module $\varphi=\otimes_{k=0}^j
\varphi_k^{[k]}\in\Phi_n$ with the following properties:
\begin{eqnarray}
\label{e1}
\varphi_k&=&M_{n,k} \mbox{ for } k\in S_1; \\
\varphi_k&\in&\mathcal{F}_n\mbox{ for } k\in S_2; \\
\varphi_k&\in&\mathcal{T}_n\mbox{ for } k\in S_3;\\
\varphi_k&\notin&\mathcal{L}_n^l\cup \mathcal{R}_n^l \mbox{ for } k\in S_2\cup S_3. \label{elast}
\end{eqnarray}
Then $\Psi\subset\Phi$.
\end{proposition}

\begin{proof}
The construction of $\Psi$ and the definition of a tensor
product of inductive systems imply that for each $\psi\in \Psi_t$
there exist $m>\max\{d,t\}$ and a $G_m$-module $\pi=\otimes_{k=0}^j
\pi_k^{[k]}$ with $\pi_k=M_{m,k}$ for $k\in S_1$,
$\pi_k\in\mathcal{F}_m$ for $k\in S_2$,  and
$\pi_k\in\mathcal{T}_m$ for $k\in S_3$, such that $\psi\in\Irr_t\pi$.
So it suffices to prove that all such modules $\pi\in\Phi_m$.
Put $l=(p-1)(m+1)$ and choose $n>m$ and $\varphi\in\Phi_n$ that
satisfies (\ref{e1})--(\ref{elast}) for this $l$. Then
Lemmas~\ref{o1} and~\ref{Hn} and Corollary~\ref{oc1} imply that $\pi_k\in\Irr_m
\varphi_k$ for all $k\in \mathbb{N}_j$. Hence $\pi\in\Irr_m
\varphi\subset\Phi_m$. This completes the proof.
\end{proof}

Note that $S_1$ can be empty.

\begin{corollary}
\label{CF} Set $n'=\left[\frac{n+1}{2}\right]$ and
$F_n=L(\omega_{n'}^n)\in\Irr G_n$. Then $\mathcal{F}= \langle
F_n \mid n\in\mathbb{N}\rangle$.
\end{corollary}

\begin{proof}
Lemma~\ref{o1} implies that $F_n\in\Irr_n F_{n+1}$. Hence
$\langle F_n \mid n\in\mathbb{N}\rangle$ is an inductive system by Lemma~\ref{par3}.
Naturally, for each $d$ there exists $n$ with
$F_n\notin\mathcal{L}_n^d \cup \mathcal{R}_n^d$. Now apply
Proposition~\ref{TF}.
\end{proof}

\begin{corollary}
\label{cp2} Define $n'$ as in Corollary~$\ref{CF}$ and set $T_n=L((p-1)\omega_{n'}^n)$.
Then $\mathcal{T}= \langle T_n \mid n\in\mathbb{N}\rangle$.
\end{corollary}

\begin{proof}
Argue as in the proof of Corollary~\ref{CF}, applying Lemmas~\ref{o2} and~\ref{par3}, and
Propositions~\ref{LR} and~\ref{TF}.
\end{proof}

\begin{proposition}
\label{t2} Let $\Phi\subset \mathcal{L}^d$ or $\mathcal{R}^d$.
Then $\Phi$ is a finite union of systems $C_L(a_1,\ldots,a_d)$ or
$C_R(a_1,\ldots,a_d)$,  respectively.
\end{proposition}

\begin{proof} We shall prove the claim for $\mathcal{L}^d$. The proof
for $\mathcal{R}^d$ is
similar. Assume that $\Phi\subset\mathcal{L}^d$. For a $d$-tuple
$s=(a_1,\ldots,a_d)$ with $a_j\in\mathbb{Z}_{\ge0}$ set
$M_{n,L}(s)=L(a_1\omega_1^n+\ldots+a_d\omega_d^n)$ and
$C_L(s)=C_L(a_1,\ldots,a_d)$. Denote by $S_d$ the set of all such
tuples with $a_1+2a_2+\ldots+d a_d\leq d$. Obviously,  the set
$S_d$ is finite. By Proposition~\ref{LR},
$\mathcal{L}_n^d=\{M_{n,L}(s)\mid s\in S_d \}$. Let $S(\Phi)=\{
s\in S_d \mid C_L(s)\subset\Phi \}$ and $\Psi=\cup_{s\in S(\Phi)}
C_L(s)$. We claim that $\Psi=\Phi$. Suppose this is not the case
and set $D_n=\Phi_n\setminus \Psi_n$. Then $D_n\ne \varnothing$
for large enough $n$. Hence there exists $\sigma\in S_d$ for
which the set $\{n\mid M_{n,L}(\sigma)\in D_n\}$ is infinite. Lemma~\ref{Hn} implies that $M_{n,L}(\sigma)\in\Irr_n M_{k,L}(\sigma)$ for $k>n$. Since $\Phi$ is an inductive system, this forces $C_L(\sigma)\subset \Phi$ and yields a contradiction. Hence $\Psi=\Phi$ as desired.
\end{proof}

\begin{corollary}
\label{ct2} If $\Phi\subset \mathcal{L}^d$ or $\mathcal{R}^d$ is an
indecomposable inductive system, then $\Phi=C_L(a_1,\ldots,a_d)$ or
$C_R(a_1,\ldots,a_d)$, respectively.
\end{corollary}

\begin{lemma}
\label{chain}
Let $\Phi\subset\mathcal{L}^a \cup\mathcal{R}^b$, but
$\Phi\not\subset\mathcal{L}^a$ and $\Phi\not\subset\mathcal{R}^b$.
Then $\Phi=\Phi^L \cup\Phi^R$ where $\Phi^L$ and $\Phi^R$ are
proper subsystems of $\Phi$, $\Phi^L\subset\mathcal{L}^a$, and
$\Phi^R\subset\mathcal{R}^b$.
\end{lemma}

\begin{proof}
Set $\Pi_n=\Phi_n\cap\mathcal{L}_n^a$,
$\Sigma_n=\Phi_n\cap\mathcal{R}_n^b$. Observe that $\Pi_n\cap
\Sigma_n=\varnothing$ for $n\geq a+b$. As  $\mathcal{L}^a$
and $\mathcal{R}^b$ are inductive systems, this implies the following: if  $n\geq a+b$, $\varphi\in\Pi_n$ or $\Sigma_n$, $\psi\in\Phi_{n+1}$, and $\varphi\in\Irr_n \psi$, then $\psi\in\Pi_{n+1}$ or $\Sigma_{n+1}$, respectively. Since $\Phi$ is an inductive system, we conclude  that  for every $\varphi\in\Pi_n$ or $\Sigma_n$ there exists $\rho\in\Pi_{n+1}$ or $\Sigma_{n+1}$, respectively, with $\varphi\in\Irr_n\rho$. Now Corollary~\ref{cgen} yields that the inductive systems
$\Phi^L=\langle \Pi_n \mid n\geq a+b \rangle$ and $\Phi^R=\langle
\Sigma_n \mid n\geq a+b \rangle$ are well defined. It is clear
that $\Phi_n=\Phi_n^L \cup\Phi_n^R$. Hence $\Phi=\Phi^L \cup\Phi^R$.
\end{proof}

Now we start describing BWM-systems for groups of type $A_n$. Note
that $\mathcal{F}=\mathcal{T}$ for $p=2$, but this does not
affect the proofs.

\begin{proposition}
\label{rgen} Let $\Phi$ be a $p$-restrictedly generated BWM-system.
Then one of the following holds:

$(1)$  $\Phi=\mathcal{F}$;

$(2)$  $\Phi=\mathcal{T}$;

$(3)$  $\Phi=\mathcal{F}\cup\mathcal{T}$;

$(4)$ $\Phi\subset \mathcal{L}^d\cup \mathcal{R}^d$;

$(5)$ $\Phi=\Phi'\cup\mathcal{T}$,  $\Phi=\Phi'\cup\mathcal{F}$, or
$\Phi=\Phi'\cup\mathcal{F}\cup\mathcal{T}$ with $\Phi'\subset
\mathcal{L}^d\cup \mathcal{R}^d$.

In all cases, if $\wdeg \Phi=k$, then $\Phi\subset \mathcal{L}^{k+2}\cup
\mathcal{R}^{k+2}\cup\mathcal{F}\cup\mathcal{T}$.
\end{proposition}

\begin{proof} Assume that $\wdeg\Phi=k$. First suppose that
$\Phi\not\subset\mathcal{F}\cup\mathcal{T}$. Then
$\Phi_n\not\subset\mathcal{F}_n\cup\mathcal{T}_n$ for large enough
$n$. Set $m=(k+1)^2 p^2$, fix  $n>m$ and a $p$-restricted
$\varphi\in\Phi_n\setminus\{ \mathcal{F}_n\cup \mathcal{T}_n\}$.
Proposition~\ref{Cp} implies that $\pdeg\varphi$ or $\pdeg\varphi^*\leq
n$ since otherwise $\wdeg\varphi>\sqrt{n}/p-1>k$. Now
Proposition~\ref{deg} forces that $\pdeg \varphi$ or $\pdeg\varphi^*\leq
k+2$ and hence $\varphi\in\mathcal{L}_n^{k+2}$ or
$\mathcal{R}_n^{k+2}$ by Proposition~\ref{LR}. This yields the
last claim of the proposition.

Now we want to reduce the problem to the situation where both $\mathcal{F}\not\subset\Phi$
and $\mathcal{T}\not\subset\Phi$. Assume that this is not the case.
Put $\Psi=\mathcal{F}$ if
$\mathcal{F}\subset\Phi$, but $\mathcal{T}\not\subset\Phi$;
$\Psi=\mathcal{T}$ if $\mathcal{T}\subset\Phi$, but
$\mathcal{F}\not\subset\Phi$; and $\Psi=\mathcal{F}\cup\mathcal{T}$
if $\mathcal{F}\cup\mathcal{T}\subset\Phi$. If $\Psi=\Phi$, the
proposition is proved. Assume that $\Psi\ne \Phi$ and put $D=D(\Phi,\Psi)$.
We claim that both $\mathcal{F}\not\subset D$ and $\mathcal{T}\not\subset D$.

If $\Psi\ne \mathcal{T}\cup\mathcal{F}$, define an inductive system
$D'$ by the equality $\{\Psi,D'\}=\{\mathcal{F},\mathcal{T}\}$. The arguments in the first paragraph of the proof yield that if $n>m$ and $\varphi\in(\Phi_n\setminus\Psi_n)$, then $\varphi\in\mathcal{L}_n^{k+2}\cup \mathcal{R}_n^{k+2}$ or
$\mathcal{L}_n^{k+2}\cup \mathcal{R}_n^{k+2}\cup D'_n$. Since
$D=\langle \Phi_n\setminus\Psi_n\mid n>m\rangle$, we observe that
$D\subset\mathcal{L}^{k+2}\cup \mathcal{R}^{k+2}$ or
$D\subset\mathcal{L}^{k+2}\cup \mathcal{R}^{k+2}\cup D'$ which
yields our claim. Replacing $\Phi$ by $D$ if necessary, we assume
that both $\mathcal{F}\not \subset\Phi$ and $\mathcal{T}\not \subset\Phi$.

Proposition~\ref{TF} implies that for some $l$ the intersections
$\Phi_n\cap\mathcal{F}_n$ and $\Phi_n\cap\mathcal{T}_n\subset
\mathcal{L}_n^l\cup \mathcal{R}_n^l$ for all $n$. Put
$d=\max(l,k+2)$. Then the last claim of the proposition implies that $\Phi_n\subset \mathcal{L}_n^d\cup
\mathcal{R}_n^d$ and hence $\Phi\subset \mathcal{L}^d\cup\mathcal{R}^d$.
\end{proof}

\begin{proof}[Proof of Theorem~$\ref{t1A}$]
The theorem follows immediately from Propositions~\ref{t2} and~\ref{rgen},
Corollaries~\ref{CF}, \ref{cp2}, and~\ref{ct2},
 and Lemma~\ref{chain}.
\end{proof}

Let $\Phi$ be an inductive system  with $\delta(\Phi)<p^{j+1}$ for
some $j\in\mathbb{Z}_{\ge0}$. Then each $\varphi\in\Phi_n$ can be uniquely
represented in the form $\otimes_{k=0}^j\varphi_k^{[k]}$ with
$\varphi_k\in\Irr^p G_n$. This notation is used in
Proposition~\ref{pp}.

\begin{proposition}
\label{pp} Let $\Phi$ be a BWM-system with $\delta(\Phi)<p^{j+1}$.
Then there exists an integer $N=N(\wdeg\Phi, j)$ with the
following properties: if $d\geq N$, $U_1$,
$U_2\subset\mathbb{N}_j$, $U_1\cap U_2=\varnothing$,
$U_2=\varnothing$ for $p=2$, $\varphi\in\Phi_n$,
$\varphi_k\in\mathcal{F}_n$ for all $k\in U_1$,
$\varphi_k\in\mathcal{T}_n$ for all $k\in U_2$,
$\varphi_k\notin\mathcal{L}_n^d \cup \mathcal{R}_n^d$ for each $k\in
U_1 \cup U_2$, and $\varphi\in\Irr_n\psi$ with $\psi\in\Phi_q$,
$q>n$, then $\psi_k\in \mathcal{F}_q$ for $k\in U_1$,
$\psi_k\in\mathcal{T}_q$ for $k\in U_2$, and $\psi_k\notin
\mathcal{L}_q^d\cup\mathcal{R}_q^d$ for $k\in U_1\cup U_2$.
\end{proposition}

\begin{proof}
Let $\wdeg\Phi=c$. Proposition~\ref{tenz}  yields that for all $a$, $b\in\mathbb{N}$
there exists $t=t(a,b)$ such that the following holds: if $n>t$,
$M=\otimes_{k=0}^sM_k^{[k]}$ with $M_k\in\Irr^p G_n$,
all $M_k\in\mathcal{L}_n^a$ or all $M_k\in\mathcal{R}_n^a$,
$\delta(M)\geq p^{s+1}$, $F\in\mathcal{F}_n$ or
$\mathcal{T}_n$, and $F\notin\mathcal{L}_n^t$ or
$\mathcal{R}_n^t$, respectively, then $\wdeg(M\otimes(F^{[s+1]}))>b$.
One may assume that $t(a,b)\geq a+2b$. Now fix
\begin{equation}
\label{e2} t_1=t(c+2,c) \mbox{ and } t_k=t(t_{k-1},c) \mbox{ for }
1<k\leq j.
\end{equation}
Hence
\[
t_j>\ldots>t_1\geq 3c+2.
\]
Set $g=c+2+\sum_{k=1}^j t_k p^k$ and $N=\max(g,(c+1)^2 p^2+1)$.

Let $n>N$, $\varphi\in\Phi_n$ and satisfy the assumptions of the
proposition with this $N$ and some $d\geq N$. Assume that
$\psi\in\Phi_q$ and $\varphi\in\Irr_n\psi$. Arguing as in the first paragraph of the proof of
Proposition~\ref{rgen}, one can conclude that for all $k$
\begin{equation}
\label{e3} \psi_k\in\mathcal{L}_q^{c+2} \cup
\mathcal{R}_q^{c+2}\cup \mathcal{F}\cup\mathcal{T}.
\end{equation}

We claim that $\varphi_k\in\Irr_n\psi_k$ for $k\in U_1\cup U_2$. To prove this,
we shall show that $\delta(\otimes_{s=0}^{k-1} \psi_s^{[s]})<p^k$ if
$k\in U_1\cup U_2$ and $k>0$.
For $k>0$ and $l<k$ put $\pi(l,k)=\otimes_{s=l}^{k-1} \psi_s^{[s]}$,
$\pi(k)=\pi(0,k)$, $\rho(l,k)=\otimes_{s=l}^k
\psi_s^{[s]}$, and $\rho(k)=\rho(0,k)$. Assume that $\delta(\pi(k))\geq p^k$ for some $k\in U_1\cup U_2$. If there exists $i<k$ with $\delta(\pi(i))<p^i$, choose
maximal such $i$ and put $l=i$. Otherwise put $l=0$. Then
$\delta(\pi(l,k))\geq p^k$. One easily observes that
$\delta(\psi_l)\geq p$ since otherwise $\delta(\pi(l+1))< p^{l+1}$,
which contradicts the choice of $l$. Hence
$\omega(\psi_l)\notin\Omega_p(G_q)$. So
$\psi_l\in\mathcal{L}_q^{c+2} \cup \mathcal{R}_q^{c+2}$ by~(\ref{e3}).

Assume that $\psi_l\in\mathcal{L}_q^{c+2}$. Put $f_u=t_{u-l}$ for $l<u\leq
k$. We claim that $\psi_u\in\mathcal{L}_q^{f_u}$ for such $u$.
Using~(\ref{e3}), we conclude that $\psi_u\in\mathcal{L}_q^{c+2}\cup
\mathcal{R}_q^{c+2}$ if $\omega(\psi_u)\notin\Omega_p(G_q)$.
Recall that $t_{u-l}\geq t_1\geq 3c+2$. First let $u=l+1$. Obviously, $\psi_u\in \mathcal{L}_q^{t_1}$ if  $\psi_u\in\mathcal{L}_q^{c+2}$. Observe that $\wdeg \rho(l,u)=\wdeg (\psi_l \otimes (\psi_u^{[1]}))\leq c$.
Since $n>N>t_1\geq 3c+2$ and hence $n\geq3c+4$, Proposition~\ref{tenz} yields that $\psi_u\not\in\mathcal{R}_q^{c+2}$ if
$\omega(\psi_u)\neq 0$. Let $\psi_u\in\mathcal{F}_q \cup\mathcal{T}_q$. Then Formula~(\ref{e2}) and the arguments above that formula yield that $\psi_u\in\mathcal{L}_q^{t_1}$. This completes the proof of the claim for $u=l+1$.

Now assume that $u>l+1$ and
apply induction on $u$. Suppose that $\psi_s\in\mathcal{L}^{f_s}$
for $l<s<u$. Then $\psi_s\in\mathcal{L}^{f_{u-1}}$ for these $s$ as
$f_s<f_{u-1}$ if $s<u-1$. The choice of $l$ shows that
$\delta(\pi(l,u))\geq p^u$ since otherwise $\delta(\pi(u))<p^u$,
which yields a contradiction. Write $\rho(l,u)=\rho'^{[l]}$ and
observe that $\wdeg \rho'=\wdeg\rho(l,u)\leq c$. Applying
Proposition~\ref{tenz} and arguing as above, we conclude that
$\psi_u\notin\mathcal{R}_q^{c+2}$ and
$\psi_u\in\mathcal{L}_q^{f_u}$ if $\psi_u\in\mathcal{F}_q \cup
\mathcal{T}_q$. Here it is essential that $n>t_j>t_{u-l}\geq
t_{u-1-l}+2c$ and so $n>t_{u-1-l}+2c+2$. Put $g'=c+2+\sum_{h=1}^{k-l} t_h p^h$. Then
for $u=k$ one has $\rho'\in\mathcal{L}_q^{g'}$. Obviously,  $g'\leq g$ (the equality
holds only for $l=0$ and $k=j$).

If $\psi_l\in \mathcal{R}^{c+2}_q$,
similar arguments yield that $\rho(l,k)=\rho'^{[l]}$ with
$\rho'\in\mathcal{R}_q^{g'}$. Using the Steinberg tensor product theorem,
we conclude that $\otimes_{s=0}^{l-1} \varphi_s ^{[s]} \in\Irr_n \pi(l)$ if $l>0$
and in all  cases there exists $\mu\in\Irr_n \rho'$ with $\mu=
(\otimes_{s=0}^{k-l} \varphi_{s+l} ^{[s]}) \otimes (\chi ^{[k-l+1]})$,
$\chi\in\Irr G_n$. Since $\mathcal{L}^{g'}$ and
$\mathcal{R}^{g'}$ are inductive systems, this forces $\varphi_k\in
\mathcal{L}_n^{g'}$ or $\mathcal{R}_n^{g'}$ and yields a
contradiction as $\mathcal{L}^{g'} \cup \mathcal{R}^{g'} \subset
\mathcal{L}^N \cup \mathcal{R}^N$. Hence $\delta(\pi(k))<p^k$ if
$k>0$ and $k\in U_1 \cup U_2$.

For $k=0$ it follows from the Steinberg tensor product theorem that there exists
$\mu\in\Irr_n \psi_k$ with $\mu=\varphi_k \otimes (\mu'^{[1]})$,
where $\mu'\in\Irr G_n$. Since $\delta(\pi(k))<p^k$ if $k>0$ and $k\in U_1\cup U_2$,
one can conclude that the same holds for all such $k$. Obviously,
$\mu=\varphi_k$ if $\omega(\psi_k)\in\Omega_p(G_n)$. Assume this is not the
case. Then $\psi_k\in \mathcal{L}_q^{c+2} \cup \mathcal{R}_q^{c+2}$ by~(\ref{e3}).
But then $\varphi_k\in\mathcal{L}_n^{c+2} \cup \mathcal{R}_n^{c+2} \subset \mathcal{L}_n^N
\cup \mathcal{R}_n^N$ which yields a contradiction. Hence
$\psi_k\in\mathcal{F}^q \cup\mathcal{T}^q$ and $\varphi_k\in\Irr_n
\psi_k$. Naturally, $\psi_k\notin\mathcal{L}_q^d \cup
\mathcal{R}_q^d$ as otherwise  $\varphi_k\in\mathcal{L}_n^d \cup \mathcal{R}_n^d$ since $\mathcal{L}^d$ and $\mathcal{R}^d$ are inductive systems.
Now Lemmas~\ref{o1} and~\ref{o2} imply that $\psi_k\in\mathcal{F}^q$
if $\varphi_k\in\mathcal{F}^q$ and $\psi_k\in\mathcal{T}^q$ if
$\varphi_k\in \mathcal{T}^q$. This completes the proof.
\end{proof}

\begin{proof}[Proof of Theorem~$\ref{a}$] {\bf (1) Indecomposable systems.}
Recall that an inductive system $\Phi=\otimes_{k=0}^j
\Fr^k(\Phi^k)$ is special if each $\Phi^k=C_L(a_1,\ldots,a_s)$,
$C_R(a_1,\ldots,a_s)$, $\mathcal{F}$, or $\mathcal{T}$. Let $\Phi$ be special. We can
write $\Phi=\otimes_{f=0}^l \Psi^f$, where $\Psi^f$ are determined
as before the statement of this theorem in the Introduction. Define the parameters $i_f$
with $0\leq f\leq l$ as in~(\ref{Fr1ind}). Let $\delta(\Psi^f)<p^{i_f+1}$ for all $f<l$. If all systems
$\Phi^k\in\{\mathcal{F}, \mathcal{T}\}$, it is clear that $\wdeg \varphi=1$ for every  $\varphi\in\Phi_n$. Otherwise one can conclude that for some $d$ and $N\in\mathbb{N}$ the system $\Phi$ is generated by a collection
$\{R_n\mid n\geq N\}$ that consists of representations satisfying the assumptions of Theorem~\ref{CnonpA} for this $d$. Now Theorem~\ref{CnonpA} and Proposition~\ref{deg*} imply that
$\Phi$ is a BWM-system if $\delta(\Psi^f)<p^{i_f+1}$ for all $f<l$.

Next, suppose  that $\delta(\Psi^f)\geq p^{i_f+1}$ for some $f<l$. The definition of the systems $\Psi^f$ implies that one of the following holds: 
\begin{itemize}
\item[(a)] $\Phi^k=C_L(a_{1,k},\ldots,a_{d_k,k})$ for $i_{f-1}+1\leq k\leq i_f$ and $\Phi^{i_f+1}=C_R(b_1,\ldots, b_t)$, $\mathcal{F}$, or $\mathcal{T}$; 
\item[(b)] $\Phi^k=C_R(a_{1,k},\ldots,a_{d_k,k})$ for $i_{f-1}+1\leq k\leq i_f$ and $\Phi^{i_f+1}=C_L(b_1,\ldots, b_t)$, $\mathcal{F}$, or $\mathcal{T}$. 
\end{itemize}
Here $0\leq a_{i,j}<p$, $0\leq b_m<p$, and $\Phi^{i_f+1}$ is nontrivial. 
Consider Case (a). Set $i=i_{f-1}+1$, $h=i_f-i$, and $q=\max\{d_k\mid i\leq k\leq i_f\}$.  Let $n>q+t$ if $\Phi^{i_f+1}=C_R(b_1,\ldots, b_t)$ and $n>q+1$ otherwise. Put $M^u_n=M_{n,L}(a_{1,u+i},\ldots,a_{d_{u+i},u+i})$ for $0\leq u\leq h$ and $M_n=\otimes_{u=0}^h (M^u_n)^{[u]}$. Set $T_n=M_{n,R}(b_1,\ldots, b_t)$ if $\Phi^{i_f+1}=C_R(b_1,\ldots, b_t)$ and $T_n=L(\omega_n^n)$ otherwise. Let $Q_n=M_n \otimes T_n^{[h+1]}$. Obviously, $L(\omega_n^n)\in \mathcal{F}_n$ and $\mathcal{T}_n$. Hence in all cases
\[
Q_n^{[i]}\in \left(\otimes_{k=i}^{i_f+1} \Fr^k(\Phi^k)\right)_n.
\]
So if $f>0$, the set $\Phi_n$ contains a module of the form $L_n\otimes Q_n^{[i]}\otimes S_n^{[i_f+2]}$ with $L_n$, $S_n\in\Irr G_n$ and $\omega(L_n)=c_1\omega_1^n+\ldots+c_n\omega_n^n$ with $c_y<p^i$ (the module $S_n$ is trivial if $i_f+1=j$). If $f=0$, then $\Phi_n$ contains a module of the form $Q_n^{[i]}\otimes S_n^{[i_f+2]}$.

Now we estimate $\wdeg Q_n$. It is clear that
\[
\delta(M_n)=\sum_{u=0}^h p^u(a_{1,u+i}+\ldots+a_{d_{u+i},u+i}).
\]
It follows from the construction of the system $\Psi^f$ that $\delta(\Psi^f)=p^i\delta(M_n)$. Hence $\delta(M_n)\geq p^{h+1}$. Obviously, $\omega(M_n)=\sum_{r=1}^q g_r\omega_r^n$ and $\omega(Q_n)=\sum_{r=n-t+1}^n m_r\omega_r^n$ if $\Phi^{i_f+1}=C_L(b_1,\ldots, b_t)$. Hence Proposition~\ref{tenz} yields that $\wdeg Q_n\geq n-t-q$ if $\Phi^{i_f+1}=C_L(b_1,\ldots, b_t)$ and $\wdeg Q_n\geq n-q-1$ otherwise. So $\wdeg Q_n$ is not bounded. Now Lemma~\ref{degtenz} implies that $\Phi$ is not a BWM-system. In Case (b) the arguments are similar.

Lemmas~\ref{Hn} and~\ref{LL} and Corollaries~\ref{CF} and~\ref{cp2} yield that each special inductive system $\Phi$ has the form $\Phi=\langle \varphi_n \mid n\geq A\rangle$ where $\varphi_n\in\Irr G_n$, $A\in\mathbb{N}$. Hence
special systems are indecomposable.

Now we will show that every
indecomposable BWM-system is a special system with
$\delta(\Psi^f)<p^{i_f+1}$ for $f<l$.
Let $\Phi$ be an indecomposable inductive system and $\wdeg\Phi=c$. By
Theorem~\ref{anSt},  $\Phi=\otimes_{k=0}^j \Fr^k(\Phi^k)$, where
$\Phi^k$ are $p$-restrictedly generated inductive systems. It follows
from Lemma~\ref{degtenz} that $\wdeg \Phi^k\leq c$ for $0\leq k\leq
j$. One easily concludes that $\Phi^k$ are indecomposable. By
Theorem~\ref{t1A}, each $\Phi^k=C_L(a_1,\ldots,a_d)$,
$C_R(a_1,\ldots,a_d)$, $\mathcal{F}$, or $\mathcal{T}$, i.e. $\Phi$ is special. This
completes the proof of the theorem for indecomposable systems.

{\bf (2) Arbitrary systems.} Let $\mathcal{B}$ be an arbitrary BWM-system. We describe a procedure that allows one either to show that $\mathcal{B}\subset\mathcal{L}^d\cup \mathcal{R}^d$ for some $d$, or to construct explicitly a subsystem $\mathcal{S}\subset \mathcal{B}$ such that $\mathcal{S}$ is a finite union of indecomposable inductive systems and $D(\mathcal{B},\mathcal{S})\subset \mathcal{L}^d\cup \mathcal{R}^d$. Then Proposition~\ref{t2} and Lemma~\ref{chain} imply that $\mathcal{B}$ is a finite union of indecomposable BWM-systems.

Fix minimal $j$ with $\delta(\mathcal{B})<p^{j+1}$. Then for all $n$ and each
$\varphi\in\mathcal{B}_n$ we have $\varphi=\otimes_{k=0}^j \varphi_k ^{[k]}$
with $\varphi_k\in\Irr^p G_n$. Until the end of this proof for a
module  $\psi\in\mathcal{B}_n$ we denote by $\psi_k$, $0\leq
k \leq j$, the modules  in $\Irr^p G_n$ that occur in such
decomposition. Set
\[
\Delta_{n,k}=\{M\in\Irr^p G_n \mid M=\varphi_k \mbox{ for some }
\varphi\in\Phi_n \}, \quad 0\leq k\leq j.
\]
Assume that $\wdeg\mathcal{B}=c$. By Lemma~\ref{degtenz},
$\wdeg M\leq c$ for all $M\in\Delta_{n,k}$. Arguing as in the
proof of Proposition~\ref{rgen}, one concludes that
\begin{equation}
\label{e5} \Delta_{n,k}\subset \mathcal{F}_n \cup\mathcal{T}_n \cup
\mathcal{L}_n^{c+2} \cup \mathcal{R}_n^{c+2}
\end{equation}
for $n>(c+1)^2 p^2$ and $0\leq k\leq j$. First assume that
\begin{equation}
\label{e6} \mbox{for every $d$ there exist $n$ and $k$ with }
\Delta_{n,k}\cap(\mathcal{F}_n \cup \mathcal{T}_n)\not\subset
\mathcal{L}_n^d \cup \mathcal{R}_n^d.
\end{equation}
If $p\ne 2$,  denote by $\mathcal{C}$ the collection of pairs $(V_1,V_2)$, $V_i
\subset\mathbb{N}_j$ with the following properties:
\begin{itemize}
\item[(i)]
$V_1\cap V_2=\varnothing$, $V_1\cup V_2\ne \varnothing$;
\item[(ii)]
for each $d$ there exist $n$ and $\varphi\in\mathcal{B}_n$ such
that $\varphi_k\in\mathcal{F}_n$ for $k\in V_1$,
$\varphi_k\in\mathcal{T}_n$ for $k\in V_2$, and
$\varphi_k\notin\mathcal{L}_n^d \cup\mathcal{R}_n^d$ for $k\in V_1\cup
V_2$;
\item[(iii)]
there is no pair $(V'_1,V'_2)$ such that $V'_1$ and $V'_2$
satisfy (i) and (ii), $V_1\subset V'_1$, $V_2\subset
V'_2$, and $V'_1 \cup V'_2\neq V_1 \cup V_2$.
\end{itemize}
As $\mathcal{L}_n^d \cup \mathcal{R}_n^d\subset \mathcal{L}_n^m \cup \mathcal{R}_n^m$ if $d<m$, Formula~(\ref{e6}) yields that for certain fixed $k$ the following holds: for each $d$ there exists $n$ with $\Delta_{n,k}\cap\mathcal{F}_n\not\subset \mathcal{L}_n^d \cup \mathcal{R}_n^d$ or for each $d$ there exists $n$ with $\Delta_{n,k}\cap\mathcal{T}_n\not\subset \mathcal{L}_n^d \cup \mathcal{R}_n^d$. So $\mathcal{C}$ is nonempty.

If $(V_1,V_2)\in\mathcal{C}$ and $V_1\cup V_2=\mathbb{N}_j$, set
$\Psi(V_1,V_2)=(\otimes_{k\in V_1} \Fr^k(\mathcal{F})) \otimes
(\otimes_{k\in V_2} \Fr^k(\mathcal{T}))$. Assume that
$(V_1,V_2)\in\mathcal{C}$ and $V_1\cup V_2\ne \mathbb{N}_j$. Set
$V_0=\mathbb{N}_j \setminus (V_1\cup V_2)$. Fix $t\in V_0$.

The construction of $\mathcal{C}$ implies that there exist $u=u(t)$ with the following properties: if $\varphi\in\Phi_n$, $\varphi_k\in\mathcal{F}_n$ for $k\in V_1$, $\varphi_k\in\mathcal{T}_n$
for $k\in V_2$, $\varphi_k\notin \mathcal{L}_n^u \cup\mathcal{R}_n^u$
for $k\in V_1\cup V_2$, and
$\varphi_t\in\mathcal{F}_n\cup\mathcal{T}_n$, then $\varphi_t\in
\mathcal{L}_n^u \cup \mathcal{R}_n^u$ (otherwise (iii) would not hold for $(V_1,V_2)$).
These arguments and Formulas~(\ref{e5})
and (\ref{e6}) yield that there exists $d$ such that
$\varphi_k\in\mathcal{L}_n^d \cup \mathcal{R}_n^d$ if $\varphi\in\Phi_n$,
$n>(c+1)^2 p^2$, $k\in V_0$, $\varphi_a \in\mathcal{F}_n
\setminus(\mathcal{L}_n^d \cup \mathcal{R}_n^d)$ for all $a\in V_1$,
and $\varphi_b\in\mathcal{T}_n \setminus (\mathcal{L}_n^d \cup
\mathcal{R}_n^d)$ for all $b\in V_2$. Naturally, we can enlarge $d$ and guarantee that $n>(c+1)^2 p^2$ if $\varphi_s\not\in(\mathcal{L}_n^d \cup \mathcal{R}_n^d)$ for some $s$. Denote by $S=S(V_1,V_2)$ the
set of all inductive systems $\Pi=\otimes_{k\in V_0} \Fr^k(\Pi^k)$
with the following properties: $\Pi^k=C_L(a_{1k},\ldots, a_{dk})$ or
$C_R(a_{1k},\ldots, a_{dk})$, $0\leq a_{ik}<p$, $\Pi^k\subset
\mathcal{L}^d$ or $\mathcal{R}^d$, and for each $m$ there exist $n$
and $\varphi\in\Phi_n$ with $\varphi_k=M_{n,L}(a_{1k},\ldots, a_{dk})$
or $M_{n,R}(a_{1k},\ldots, a_{dk})$ if $k\in V_0$ and $\Pi^k=C_L(a_{1k},\ldots, a_{dk})$
 or $C_R(a_{1k},\ldots, a_{dk})$, respectively,
$\varphi_k\notin\mathcal{L}_n^m \cup \mathcal{R}_n^m$ for $k\in V_1\cup
V_2$, $\varphi_k\in\mathcal{F}_n$ for $k\in V_1$, and
$\varphi_k\in\mathcal{T}_n$ for $k\in V_2$. Since the number of inductive systems $C_L(a_{1k},\ldots, a_{dk})\subset \mathcal{L}^d$ and  $C_R(a_{1k},\ldots, a_{dk})\subset \mathcal{R}^d$ is finite and $(V_1,V_2)$ satisfies the assumptions (i)-(iii), one can observe that $S$ is nonempty and finite. For $\Pi\in S$ set
\[
\Psi(\Pi)=\Pi\otimes
(\otimes_{k\in V_1} \Fr^k(\mathcal{F})) \otimes (\otimes_{k\in V_2}
\Fr^k(\mathcal{T})).
\]
 Put $\Psi(V_1,V_2)=\cup_{\Pi\in S} \Psi(\Pi)$
and $\Psi=\cup_{(V_1,V_2)\in \mathcal{C}} \Psi(V_1,V_2)$.
Proposition~\ref{TF} implies that $\Psi(V_1,V_2)\subset \mathcal{B}$
if $V_1\cup V_2=\mathbb{N}_j$ and $\Psi(\Pi)\subset\mathcal{B}$ for
all $\Pi\in S(V_1,V_2)$ if $V_1\cup V_2\ne \mathbb{N}_j$. Hence
$\Psi\subset \mathcal{B}$.

For $p=2$ let $\mathcal{C}$ be the collection of all nonempty sets
$V$ such that for each $d$ there exist $n$ and $\varphi\in
\mathcal{B}_n$ with $\varphi_k\in\mathcal{F}_n$ for $k\in V$ and $V$ is
a maximal subset in $\mathbb{N}_j$ with this property. Using Formula~(\ref{e6}) as for $p>2$, we conclude that $\mathcal{C}$ is nonempty. If $\mathcal{C}$ consists of the set $\mathbb{N}_j$, put
$\Psi=\otimes_{k=0}^j \Fr^k(\mathcal{F})$. Assume this is not the
case. For each $V\in \mathcal{C}$ construct the set $S(V)$ and the
system $\Psi(V)$ in the same way as we have constructed the sets
$S(V_1,V_2)$ and the systems $\Psi(V_1,V_2)$ for $p\ne 2$. Put
$\Psi=\cup_{V\in\mathcal{C}} \Psi(V)$. Using Proposition~\ref{TF} as
before, one concludes that $\Psi\subset \mathcal{B}$ for $p=2$ as
well. It is clear that in all cases $\Psi$ is a finite union of
indecomposable BWM-systems. So we are done if $\Psi=\mathcal{B}$.

Assume that $\Psi\ne \mathcal{B}$ and set
$\mathcal{B}^1=D(\mathcal{B},\Psi)$. Obviously, $\wdeg \mathcal{B}^1
\leq c$. Denote by $\Delta_{n,k}^1$ the analogues of the sets
$\Delta_{n,k}$ for the system $\mathcal{B}^1$. It is clear
that~(\ref{e5}) holds for $\Delta_{n,k}^1$.

Assume that~(\ref{e6}) holds for $\Delta_{n,k}^1$. Then one can
define the collection $\mathcal{C}^1$ for the system $\mathcal{B}^1$
in the same way as we have defined $\mathcal{C}$ for $\mathcal{B}$.
Put $q(\mathcal{C})=\max \{ |V_1 \cup V_2| \mid
(V_1,V_2)\in\mathcal{C} \}$ for $p>2$, $q(\mathcal{C})=\max \{ |V|
\mid V\in\mathcal{C} \}$ for $p=2$, and define $q(\mathcal{C}^1)$
similarly. We claim that $q(\mathcal{C}^1)< q(\mathcal{C})$. Indeed,
let $p>2$ and $(U_1,U_2)\in\mathcal{C}^1$. We will show that there
exists a pair $(V_1,V_2)\in\mathcal{C}$ with $U_i\subset V_i$ and
$|V_1\cup V_2|> |U_1\cup U_2|$. First we will prove that
$(U_1,U_2)\notin\mathcal{C}$. Suppose that $(U_1,U_2)\in\mathcal{C}$
for some pair $(U_1,U_2)\in\mathcal{C}^1$. Let $U_1\cup U_2\ne
\mathbb{N}_j$. The  construction of the  subsystem
$\Psi(U_1,U_2)\subset\Psi$ above yields  that for some $m=m(U_1,U_2)$ if
$\varphi\in \Phi_n$, $\varphi_k\in\mathcal{F}_n$ for all $k\in U_1$,
$\varphi_k\in\mathcal{T}_n$ for every  $k\in U_2$, and
$\varphi_k\notin\mathcal{L}_n^m \cup \mathcal{R}_n^m$ for each $k\in
U_1 \cup U_2$, then $\varphi\in \Psi(U_1,U_2)_n$.

Let $N=N(c,j)$ be such as in Proposition~\ref{pp}. Let
$d\geq N$ if $U_1\cup U_2=\mathbb{N}_j$ and $d\geq \max \{N,
m(U_1,U_2) \}$ otherwise. Since $(U_1,U_2)\in\mathcal{C}^1$, some
$\mathcal{B}^1_n$ contains a representation $\varphi$ such that
$\varphi_k\in\mathcal{F}_n$ for $k\in U_1$, $\varphi_k\in\mathcal{T}_n$
for $k\in U_2$, and $\varphi_k\not\in\mathcal{L}_n^d \cup \mathcal{R}_n^d$
for each $k\in U_1\cup U_2$. The construction of $\mathcal{B}^1$
implies that for some $t>n$ there exists a representation
$\rho\in\mathcal{B}_t\setminus\Psi_t$ with $\varphi\in \Irr_n \rho$. By
Proposition~\ref{pp}, $\rho_k\in \mathcal{F}_t$ for $k\in U_1$,
$\rho_k\in \mathcal{T}_t$ for $k\in U_2$, and $\rho_k\notin
\mathcal{L}_t^d \cup \mathcal{R}_t^d$ for $k\in U_1 \cup U_2$. This
yields a contradiction. Indeed, if $U_1\cup U_2\ne \mathbb{N}_j$,
all such representations $\rho\in\Psi(U_1,U_2)_t$ by the arguments
above. If $U_1\cup U_2=\mathbb{N}_j$, the construction of
$\Psi(U_1,U_2)$ implies that for $\rho\notin\Psi(U_1,U_2)_t$ some
$\rho_k\notin\mathcal{F}_t$ with $k\in U_1$ or some
$\rho_s\notin\mathcal{T}_t$ for $s\in U_2$. Observe that in all
cases $\Psi(U_1,U_2)\subset\Psi$. Hence
$(U_1,U_2)\notin\mathcal{C}$.

The construction of $\mathcal{C}$ and $\mathcal{C}^1$ implies that
the pair $(U_1,U_2)$ satisfies the assumptions (i) and (ii) that we
used to define $\mathcal{C}$, but does not satisfy (iii). Hence there
exists a pair $(U'_1,U'_2)$ mentioned in (iii).

Take for $(V_1,V_2)$ such pair with the maximal $|U'_1 \cup U'_2|$.
For $p=2$ similar arguments yield that each
$U\subset\mathcal{C}^1$ is the proper subset of some
$M\subset\mathcal{C}$. Hence in all cases
$q(\mathcal{C}^1)<q(\mathcal{C})$.

Now construct an inductive system $\Psi^1\subset\mathcal{B}^1$ in the
same way as $\Psi$ was constructed for $\mathcal{B}$. If
$\Psi^1\ne \mathcal{B}^1$, set
$\mathcal{B}^2=D(\mathcal{B}^1,\Psi^1)$. Continue the process until
this is possible, constructing for a system $\mathcal{B}^i$ the
collection $\mathcal{C}^i$ and the subsystem $\Psi^i$ in the same
way as $\mathcal{C}^1$ and $\Psi^1$ were constructed. By the
arguments above, if $\mathcal{C}^i$ is determined, then
$q(\mathcal{C}^i)<q(\mathcal{C}^{i-1})< \ldots < q(\mathcal{C})$.
Hence for some $i$ either $\Psi^i=\mathcal{B}^i$ or~(\ref{e6})
does not hold for $\mathcal{B}^{i+1}$. Here our procedure is finished. In the first case
$\mathcal{B}=\Psi\cup(\cup_{1\leq k\leq i} \Psi^k)$ and hence is a
finite union of indecomposable BWM-systems. Now assume
that~(\ref{e6}) does not hold for $\mathcal{B}$ or
$\mathcal{B}^{i+1}$. Set $\Sigma=\mathcal{B}$ or
$\mathcal{B}^{i+1}$, respectively. As $\Sigma$ is an inductive system, Formula~(\ref{e5}) yields that
$\Sigma\subset\mathcal{L}^d \cup\mathcal{R}^d$ for some $d$. Therefore our goal is reached. The theorem is proved.
\end{proof}

%All the results of the paper concerning the groups $A_n(K)$ are proved.

\section{Inductive systems with bounded weight multiplicities for symplectic and spinor groups}
\label{inductive2}

In this section $G_n=B_n(K)$, $C_n(K)$, or $D_n(K)$. Recall the collections $\mathcal{S}$ and $\mathcal{L}$ defined in the Introduction.
By Lemma~\ref{l*}, $\mathcal{L}$ is an inductive system in all cases.

\begin{lemma}
\label{ind2} Let $p>2$ for $G_n\ne D_n(K)$. The collection $\mathcal{S}$
is an inductive system.
\end{lemma}

\begin{proof}
This follows from Lemma~\ref{lBD} for  $G_n=B_n(K)$ or $D_n(K)$ and Lemma~\ref{rC}
for $G_n=C_n(K)$.
\end{proof}

Now we state our results on the BWM-systems in the special case where $p=2$ and $G_n=C_n(K)$. These assumptions on $p$ and $G_n$ are valid until the proof of Theorems~$\ref{t1}$ and~$\ref{12new2}$.

Set $\mathcal{S}'_n=\{L(\omega^n_n)\}$, $\mathcal{S}'=\{\mathcal{S}'_n\}_{n\in\mathbb{N}}$,
\[
\mathcal{Q}_n=\{L(\omega_1^n+\omega_n^n),L(\omega_n^n)\}
\]
for $n>1$, $\mathcal{Q}_1=\Irr_1 \mathcal{Q}_2$, and  $\mathcal{Q}=\{\mathcal{Q}_n\}_{n\in\mathbb{N}}$.

\begin{lemma}
\label{ch2} Let $p=2$ and $G_n=C_n(K)$. Then $\mathcal{S}'$ and $\mathcal{Q}$ are inductive systems.
\end{lemma}

\begin{proof}
The result follows from Lemma~\ref{lBD} and Corollary~\ref{cBD}.
\end{proof}

We need some notation to describe irreducible representations of $G_n$ with small weight multiplicities.
Put
\[
\Omega_2(G_n)=\{0, \omega^n_1,\omega^n_n\} \mbox{ and } \Omega'_2(G_n)=\Omega_2\cup \{\omega^n_1+\omega^n_n\}.
\]
For any dominant weight $\omega$ of $G_n$ we can write its "$2$-adic
expansion"
\[
\omega=\lambda_0+2\lambda_1 +\ldots+2^k \lambda_k,
\]
where weights $\lambda_i$ are $2$-restricted for $0\leq i\leq k$.
This expansion is uniquely determined if we assume that $k=0$
for $\omega=0$ and $\lambda_k\ne 0$ otherwise.
Set
\[
S(\omega)=(\lambda_0,\ldots,\lambda_k).
\]
Put
$$
\Omega(G_n)=\left\{\sum_{j=0}^k 2^j\lambda_j\mid
k\ge0,\ \lambda_j\in\Omega_2(G_n),\ (\lambda_j, \lambda_{j+1})\neq(\omega^n_n, \omega^n_1)\mbox{ for } j<k \right\}
$$
and
$$
\Omega'(G_n)=\left\{\sum_{j=0}^k 2^j\lambda_j\mid
k\ge0,\ \lambda_j\in\Omega'_2(G_n)\right\}.
$$

By~\cite[Proposition~2]{ZSVesti}, $\wdeg(L(\omega))=1$ if and only if $\omega\in\Omega(G_n)$. Thus, in this case a
connection between the sets $\Omega(G_n)$ and $\Omega_p(G_n)$ is more
complicated than for other classical groups or odd $p$.

\begin{theorem}[{\cite[Theorem 2]{C_small_char}}]
\label{12new} Let $p=2$, $G_n=C_n(K)$, $n\geq 8$, and let $M\in \Irr G_n$ with
$\omega(M)\notin\Omega(G_n)$. Then the following hold:

$(i)$ if  $\omega\in\Omega'(G_n)$, the weight $\omega^n_1+\omega^n_n$ occurs in
the sequence  $S(\omega)$ exactly $l$ times, and for $0\leq j<k$
\[
(\lambda_j,\lambda_{j+1})\notin \{ (\omega^n_n,\omega^n_1),
(\omega^n_1+\omega^n_n,\omega^n_1), (\omega^n_n,\omega^n_1+\omega^n_n),
(\omega^n_1+\omega^n_n,\omega^n_1+\omega^n_n)\},
\]
then $\wdeg M=2^l$;

$(ii)$ otherwise $\wdeg M\geq n-4-[n]_4$, where $[n]_4$ is the residue
of $n$ modulo $4$; in particular, $\wdeg M\geq n-7$.
\end{theorem}

\begin{theorem}
\label{12new2}
Let $p=2$ and $G_n=C_n(K)$.
Set $\mathcal{P}=\{\mathcal{O}, \mathcal{L}, \mathcal{Q}, \mathcal{S}'\}$.  An
indecomposable inductive system $\Phi$ is a BWM-system  if and
only if $\Phi=\otimes_{j=0}^s \Fr^j(\Phi^j)$ with $\Phi^j\in
\mathcal{P}$ and $(\Phi^j,\Phi^{j+1})\notin\{(\mathcal{S}',\mathcal{L}),(\mathcal{Q},\mathcal{L}),
(\mathcal{S}',\mathcal{Q}),(\mathcal{Q},\mathcal{Q})\}$.  BWM-systems are finite unions of indecomposable ones.
\end{theorem}

Though the description of BWM-systems is more complicated for $p=2$ and $G_n=C_n(K)$, the proofs of Theorems~\ref{t1} and~\ref{12new2} are based on similar arguments. So we prove them simultaneously.
\bigskip

\begin{proof}[Proof of Theorems~$\ref{t1}$ and~$\ref{12new2}$] In this proof we say that we are in a special case if $p=2$ and $G_n=C_n(K)$ and in the general case otherwise. Assume that $n>3$. Set $\tau_n=L(0)\in\Irr G_n$  and $\lambda_n=L(\omega_1^n)$ for all three types. Put
\[
\mu_n=
\begin{cases}
L(\frac{p-1}{2}\omega_n^n) & \mbox{ for } G_n=C_n(K), p>2,\\
L(\omega_n^n) & \mbox{ otherwise}.\\
\end{cases}
\]
In the special case also set $\xi_n=L(\omega_1^n+\omega_n^n)$.

Let $\Phi$ be a $BWM$-system. Lemma~\ref{delta}
implies that there exists $l\in\mathbb{N}$ such that for all $n\in\mathbb{N}$ and each  $\varphi\in\Phi_n$ the representation $\varphi=\otimes^l_{k=0}\varphi_k^{[k]}$ with $\varphi_k\in\Irr^p G_n$, $0\leq k\leq l$. Fix such $l$. Theorems~\ref{BCD} and~\ref{12new} imply that there exists a constant $N$ such that for $n>N$ and $\varphi\in\Phi_n$ the weight $\omega(\varphi)\in\Omega(G_n)$ in the general case and $\omega(\varphi)\in\Omega'(G_n)$ in the special case.

Now we construct a collection of inductive systems for the groups $G_n$ that actually yield all indecomposable $BWM$-systems. In the general case for a triple of subsets $A,B, C\subset\mathbb{N}_l$ such that $A\cup B\cup C=\mathbb{N}_l$ and
$A\cap B=A\cap C=B\cap C=\varnothing$ put  $\pi_n(A,B,C)=\otimes^l_{k=0}\varphi_k^{[k]}$ with
$\varphi_k=\tau_n$ for $k\in A$, $\varphi_k=\lambda_n$ for $k\in B$, and $\varphi_k=\mu_n$ for $k\in C$. In the special one for a quadruple of subsets $A,B, C, D\subset\mathbb{N}_l$ such that $A\cup B\cup C\cup D=\mathbb{N}_l$ and
$U\cap V=\varnothing$ for $U,V\in\{A,B,C,D\}$ with $U\neq V$ put $\rho_n(A,B,C,D)=\otimes^l_{k=0}\varphi_k^{[k]}$ with $\varphi_k=\tau_n$ for $k\in A$, $\varphi_k=\lambda_n$ for $k\in B$,  $\varphi_k=\mu_n$ for $k\in C$, and $\varphi_k=\xi_n$ for $k\in D$.

We need some notation to expose arguments common for the both cases. Let $\mathcal{A}=(A,B,C)$, $\psi_n(\mathcal{A})=\pi_n(A,B,C)$, $\mathcal{P}=\{\mathcal{O}, \mathcal{L}, \mathcal{S}\}$ in the general case and $\mathcal{A}=(A,B,C,D)$, $\psi_n(\mathcal{A})=\rho_n(A,B,C,D)$, $\mathcal{P}=\{\mathcal{O}, \mathcal{L}, \mathcal{Q}, \mathcal{S}'\}$ in the special one where a triple $(A,B,C)$ or a quadruple $(A,B,C,D)$ satisfies the relevant assumptions above. In what follows we shall call such tuples $\mathcal{A}$ admissible tuples. Using Lemmas~\ref{rC}, \ref{l*}, and  \ref{lBD} and Corollary~\ref{cBD}, one easily observes that $\psi_n(\mathcal{A})\in\Irr_n\psi_{n+1}(\mathcal{A})$. It is clear that
\[
\delta(\psi_n(\mathcal{A}))\leq
\begin{cases}
\frac{p^{l+1}-1}{2} & \mbox{ for } G_n=C_n(K) \mbox{ and } p>2,\\
2^{l+2}-2 & \mbox{ for } G_n=C_n(K) \mbox{ and } p=2,\\
1+p+\ldots+p^l & \mbox{ otherwise}.\\
\end{cases}
\]
Hence Lemma~\ref{par3} implies that the inductive system
$\Psi(\mathcal{A})=\langle\psi_n(\mathcal{A})\mid n>3\rangle$ is well defined. Lemmas~\ref{rC}, \ref{l*}, \ref{lBD}, and~\ref{LL} and Corollary~\ref{cBD} yield that
$$
\Psi(\mathcal{A})=\otimes^l_{k=0}\Fr^k(\Psi^k),    \quad
\Psi^k\in\mathcal{P}, \quad    0\leq k\leq l,
$$
and that each inductive system
$$
\Theta=\otimes^j_{k=0}\Fr^k(\Theta^k),    \quad
\Theta^k\in\mathcal{P}, \quad    0\leq k\leq j,
$$
coincides with $\Psi(\mathcal{A})$ for some admissible tuple $\mathcal{A}$.  Hence all these systems $\Theta$ are indecomposable.

In the general case for all admissible tuples $\mathcal{A}$ one has $\wdeg \psi_n(\mathcal{A})=1$ by Theorem~\ref{one}. In the special case for fixed $\mathcal{A}=(A,B,C,D)$ and $0\leq k<l$ we shall write $X(k)=(U,V)$ with $U,V\in \{A,B,C,D\}$ if $k\in U$ and $k+1\in V$. Theorem~\ref{12new} and~\cite[Proposition~2]{ZSVesti} force that $\wdeg\psi_n(\mathcal{A})\geq n-7$ if for some $k<l$ the pair $X(k)\in\{(C,B),(D,B), (C,D),(D,D)\}$ and $\wdeg \psi_n(\mathcal{A})\leq 2^{l+1}$ otherwise. Now Proposition~\ref{deg*} yields that in the general case all systems $\Theta$ introduced above are BWM-systems and in the special one such system is a BWM-system if and only if $(\Theta^k,\Theta^{k+1})\notin\{(\mathcal{S}',\mathcal{L}),(\mathcal{Q},\mathcal{L}),
(\mathcal{S}',\mathcal{Q}),(\mathcal{Q},\mathcal{Q})\}$ for all $k<j$.

Now assume that $n>N$. We claim that for every $\varphi\in\Phi_n$ there exists an admissible tuple $\mathcal{A}$ such that
\begin{equation}
\label{eqrho}
\psi_{n+1}(\mathcal{A})\in\Phi_{n+1} \quad \mathrm{and} \quad \varphi\in\Irr_n(\psi_{n+1}(\mathcal{A})).
\end{equation}
Indeed, since $\Phi$ is an inductive system, the representation
$\varphi\in\Irr_n\chi$ for some $\chi\in\Phi_{n+2}$. One has $\chi=\otimes^l_{k=0}\chi_k^{[k]}$ with
$\chi_k\in\Omega_p(G_{n+2})$ in the general case and $\chi_k\in\Omega'_2(G_{n+2})$ in the special one, $0\leq k\leq l$.

Lemmas~\ref{rC}, \ref{l*}, and \ref{lBD} and Corollary~\ref{cBD} imply the following: $\Irr_n \chi_k\subset \Irr^p G_n$ and hence $\phi_k\in\Irr_n \chi_k$; $\chi_k\in\mathcal{L}_{n+2}$ if $\varphi_k\in\mathcal{L}_n$, $\chi_k=\lambda_{n+2}$ for $\varphi_k=\lambda_n$, $\chi_k\in\mathcal{S}_{n+2}$ if $\varphi_k\in\mathcal{S}_n$; in the special case $\chi_k\in\mathcal{Q}_{n+2}$ if $\varphi_k\in\mathcal{Q}_n$ and $\chi_k=\xi_{n+2}$ if $\varphi_k=\xi_n$.
Then another application of those lemmas permits us to find an admissible tuple $\mathcal{A}$ such that $\psi_{n+1}(\mathcal{A})\in\Irr_{n+1}\chi$ and $\varphi\in\Irr_n(\psi_{n+1}(\mathcal{A}))$. Naturally, $\psi_{n+1}(\mathcal{A})\in\Phi_{n+1}$ as $\Phi$ is an inductive system. This proves the claim.

Since the set of admissible tuples is finite, Formula~(\ref{eqrho}) yields that for every $\phi\in\Phi_n$ there exist an infinite set $S\subset \mathbb{N}$ and an admissible tuple $\mathcal{A}$ such that $S$ consists of some integers greater than $N$, $\psi_m(\mathcal{A})\in \Phi_m$ for $m\in S$, and $\phi\in\Irr_n\psi_m(\mathcal{A})$.
Define by $I$ the collection of all tuples $\mathcal{A}$ that have this property for some $\phi$ and $n$, and set $\Sigma=\bigcup_{\mathcal{A}\in I}\Psi(\mathcal{A})$. Observe that $\Sigma=\Phi$. Naturally, $\Sigma\subset\Phi$  since $\Phi$ is an inductive system and $\Psi(\mathcal{A})=\langle\psi_m(\mathcal{A})\mid m\in S \rangle$ for every admissible $\mathcal{A}$ and infinite set $S\in\mathbb{N}$. On the other hand, the construction of $\Sigma$ yields that $\Phi_n\subset\Sigma_n$ for $n>N$ as $\Sigma$ is an inductive system. This completes the proof.
\end{proof}

\end{document}